\documentclass[a4paper,english]{article}
\usepackage[T1]{fontenc}
\usepackage[latin1]{inputenc}
\usepackage{srcltx}
\usepackage{mathtools}
\usepackage{amsmath}
\usepackage{amsthm}
\usepackage{amssymb}

\makeatletter

\pdfpageheight\paperheight
\pdfpagewidth\paperwidth

\theoremstyle{plain}
\newtheorem{thm}{\protect\theoremname}
  \theoremstyle{definition}
  \newtheorem{defn}[thm]{\protect\definitionname}
  \theoremstyle{plain}
  \newtheorem{prop}[thm]{\protect\propositionname}
  \theoremstyle{remark}
  \newtheorem{rem}[thm]{\protect\remarkname}
  \theoremstyle{definition}
  \newtheorem{example}[thm]{\protect\examplename}
  \theoremstyle{plain}
  \newtheorem{lem}[thm]{\protect\lemmaname}
  \theoremstyle{plain}
  \newtheorem{cor}[thm]{\protect\corollaryname}
  \theoremstyle{plain}
  \newtheorem*{thm*}{\protect\theoremname}

\usepackage[english]{babel}
\usepackage{indentfirst}%fa iniziare i paragrafi allineati 
\usepackage{graphicx}
\usepackage{epsfig}
\usepackage{url}

\theoremstyle{plain}
\theoremstyle{remark}
\newcommand{\U}{\mathcal{U}}
\newcommand{\N}{\mathbb{N}}
\newcommand{\Z}{\mathbb{Z}}
\newcommand{\Q}{\mathbb{Q}}

\newcommand{\V}{\mathcal{V}}

\newcommand{\W}{\mathcal{W}}

\newcommand{\bigslant}[2]{{\raisebox{.2em}{$#1$}\left/\raisebox{-.2em}{$#2$}\right.}}

\usepackage{babel}
\providecommand{\definitionname}{Definition}
  \providecommand{\lemmaname}{Lemma}
  \providecommand{\propositionname}{Proposition}
\providecommand{\theoremname}{Theorem}

\usepackage{babel}
\providecommand{\corollaryname}{Corollary}
  \providecommand{\definitionname}{Definition}
  \providecommand{\examplename}{Example}
  \providecommand{\lemmaname}{Lemma}
  \providecommand{\propositionname}{Proposition}
  \providecommand{\remarkname}{Remark}
  \providecommand{\theoremname}{Theorem}
\providecommand{\theoremname}{Theorem}

\usepackage{babel}
\providecommand{\corollaryname}{Corollary}
  \providecommand{\definitionname}{Definition}
  \providecommand{\examplename}{Example}
  \providecommand{\lemmaname}{Lemma}
  \providecommand{\propositionname}{Proposition}
  \providecommand{\remarkname}{Remark}
  \providecommand{\theoremname}{Theorem}
\providecommand{\theoremname}{Theorem}

\usepackage{babel}
\providecommand{\corollaryname}{Corollary}
  \providecommand{\definitionname}{Definition}
  \providecommand{\examplename}{Example}
  \providecommand{\lemmaname}{Lemma}
  \providecommand{\propositionname}{Proposition}
  \providecommand{\remarkname}{Remark}
  \providecommand{\theoremname}{Theorem}
\providecommand{\theoremname}{Theorem}

\usepackage{babel}
\providecommand{\corollaryname}{Corollary}
  \providecommand{\definitionname}{Definition}
  \providecommand{\examplename}{Example}
  \providecommand{\lemmaname}{Lemma}
  \providecommand{\propositionname}{Proposition}
  \providecommand{\remarkname}{Remark}
\providecommand{\theoremname}{Theorem}

\usepackage{babel}
\providecommand{\corollaryname}{Corollary}
  \providecommand{\definitionname}{Definition}
  \providecommand{\examplename}{Example}
  \providecommand{\lemmaname}{Lemma}
  \providecommand{\propositionname}{Proposition}
  \providecommand{\remarkname}{Remark}
\providecommand{\theoremname}{Theorem}

\usepackage{babel}
\providecommand{\corollaryname}{Corollary}
  \providecommand{\definitionname}{Definition}
  \providecommand{\examplename}{Example}
  \providecommand{\lemmaname}{Lemma}
  \providecommand{\propositionname}{Proposition}
  \providecommand{\remarkname}{Remark}
\providecommand{\theoremname}{Theorem}

\usepackage{babel}
\providecommand{\corollaryname}{Corollary}
  \providecommand{\definitionname}{Definition}
  \providecommand{\examplename}{Example}
  \providecommand{\lemmaname}{Lemma}
  \providecommand{\propositionname}{Proposition}
  \providecommand{\remarkname}{Remark}
\providecommand{\theoremname}{Theorem}

\usepackage{babel}
\providecommand{\corollaryname}{Corollary}
  \providecommand{\definitionname}{Definition}
  \providecommand{\examplename}{Example}
  \providecommand{\lemmaname}{Lemma}
  \providecommand{\propositionname}{Proposition}
  \providecommand{\remarkname}{Remark}
\providecommand{\theoremname}{Theorem}

\usepackage{babel}
\providecommand{\corollaryname}{Corollary}
  \providecommand{\definitionname}{Definition}
  \providecommand{\examplename}{Example}
  \providecommand{\lemmaname}{Lemma}
  \providecommand{\propositionname}{Proposition}
  \providecommand{\remarkname}{Remark}
\providecommand{\theoremname}{Theorem}

\usepackage{babel}
\providecommand{\corollaryname}{Corollary}
  \providecommand{\definitionname}{Definition}
  \providecommand{\examplename}{Example}
  \providecommand{\lemmaname}{Lemma}
  \providecommand{\propositionname}{Proposition}
  \providecommand{\remarkname}{Remark}
\providecommand{\theoremname}{Theorem}

\makeatother

\usepackage{babel}
  \providecommand{\corollaryname}{Corollary}
  \providecommand{\definitionname}{Definition}
  \providecommand{\examplename}{Example}
  \providecommand{\lemmaname}{Lemma}
  \providecommand{\propositionname}{Proposition}
  \providecommand{\remarkname}{Remark}
  \providecommand{\theoremname}{Theorem}
\providecommand{\theoremname}{Theorem}

\begin{document}

\title{Ultrafilters, monads and combinatorial properties}

\author{Lorenzo Luperi Baglini\thanks{Faculty of Mathematics, University of Vienna, Oskar-Morgenstern-Platz
1, 1090 Vienna, Austria, supported by grant P 30821 of the Austrian
Science Fund FWF.}}
\maketitle
\begin{abstract}
We use nonstandard methods, based on iterated hyperextensions, to
develop applications to Ramsey theory of the theory of monads of ultrafilters.
This is performed by studying in detail arbitrary tensor products
of ultrafilters, as well as by characterizing their combinatorial
properties by means of their monads. This extends to arbitrary sets
methods previously used to study partition regular Diophantine equations
on $\N$. Several applications are described by means of multiple
examples.
\end{abstract}

\section{Introduction}

It is well known that ultrafilters and nonstandard analysis are strictly
related: on the one hand, models of nonstandard analysis are characterized,
up to isomorphisms, as limit ultrapowers (see \cite[Section 6.4]{05});
on the other hand, the correspondence between elements of a nonstandard
extension $\prescript{\ast}{}{X}$ and ultrafilters on $X$ was first
observed (in the more general case of filters) by W.A.J.~Luxemburg
in \cite{22}, who introduced the concept of monad of a filter. This
correspondence was then used by C.~Puritz in \cite{24,25} and by
G. Cherlin and J. Hirschfeld in \cite{06} to produce new results
about the Rudin-Keisler ordering and to characterize several classes
of ultrafilters, including P-points and selective ones. Similar ideas
were also pursued by S.~Ng and H.~Render in \cite{23} and by A.~Blass
in \cite{Blass2}. 

In \cite{17}, we proved a combinatorial characterization of monads
of ultrafilters in $\beta\N$ which made it possible to develop several
applications in the study of the partition regularity of Diophantine
equations\footnote{See Theorem \ref{thm: monads and polynomials} for a definition of
this notion.} by means of some rather simple algebraic manipulations of hypernatural
numbers. The partition regularity of Diophantine equations is a particular
instance of the kind of problems that are studied in Ramsey theory,
where one wants to understand which monochromatic structures can be
found in some piece of arbitrary finite partitions of a given object.

The basic idea behind our nonstandard approach to Ramsey theory is
that every set in a ultrafilter $\U\in\beta\N$ satisfies a prescribed
property $\varphi$ if and only if the monad of $\U$ satisfy an appropriate
nonstandard version of $\varphi$. This idea has been developed in
\cite{09,11,12,18,19,20,21}, and belong to the family of applications
of nonstandard analysis in Ramsey theory, an approach that started
with J. Hirschfeld in \cite{Hirsch} and has subsequently been carried
on by many authors. As R.~Jin pointed out, nonstandard methods in
Ramsey theory are very useful because they can be used to reduce the
complexity of the mathematical objects that one needs in a proof,
therefore offering a much better intuition, which allows to obtain
much simpler (and shorter) proofs. 

In \cite{10}, M.~Di Nasso surveyed the nonstandard characterization
of ultrafilters on $\N$, proving also several equivalent characterization
of the elements of the monads of tensor products of ultrafilters.
This paper can be seen as an extension of such a study, since our
main aim is to characterize monads of ultrafilters and tensor products
of ultrafilters on arbitrary sets, so to extend the nonstandard methods
used for Diophantine equations to more general classes of problems
in Ramsey theory. This requires to better understand arbitrary tensor
products of ultrafilters, which are a basic important tool to develop
such applications (e.g., in \cite{polyextensions} tensor products
of ultrafilters in $S^{n}$, for $S$ semigroup, are used to obtain
polynomial extensions of the Milliken-Taylor theorem). Moreover, it
is helpful to characterize the Ramsey-theoretical properties of monads
in terms of their combinatorial and algebraic structure for general
properties, extending what we already did for Diophantine equations;
such an approach could lead to unexpected applications in other related
fields. It turns out that a good nonstandard framework to perform
this study is given by iterated nonstandard extensions. 

In Section \ref{sec: Iterated-Hyperextensions} we recall the basic
definitions and properties of iterated hyperextensions, providing
the nonstandard framework that is used to develop the rest of the
paper. In Section \ref{sec:Monads} we recall the definition of the
monad of an ultrafilter. We also recall some basic properties of these
monads, presenting some of their peculiar properties in iterated hyperextensions.
In Section \ref{sec: tensor pairs} we consider arbitrary tensor products
of ultrafilters, we provide several equivalent characterizations of
the elements in their monads and we extend the characterizations to
tensor products of arbitrary (finite) length. Finally, in Section
\ref{sec:Combinatorial-properties-of} we present several combinatorial
properties of monads of arbitrary ultrafilters. In all the paper,
several examples are also included to illustrate the use of such a
theory in applications, as well as our main ideas. 

This paper is self-contained: we only assume the reader to know the
basics of ultrafilters and nonstandard analysis, in particular the
notions of superstructure, transfer, ultrafilter, enlarging and saturation
properties. In any case, a comprehensive reference about ultrafilters
and their applications, especially in Ramsey theory, is the monograph
\cite{13}. As for nonstandard analysis, many short but rigorous presentations
can be found in the literature. We suggest \cite{02}, where eight
different approaches to nonstandard methods are presented, as well
as the introductory book \cite{Goldblatt}, which covers all the nonstandard
tools that we need in this paper, except the iterated extensions that
we will discuss in Section \ref{sec: Iterated-Hyperextensions}. 

\section{\label{sec: Iterated-Hyperextensions}Iterated Hyperextensions}

In this paper, we will adopt the so-called ``external'' approach
to nonstandard analysis, based on superstructure models of nonstandard
methods (see also \cite[Section 3]{02}): 
\begin{defn}
A superstructure model of nonstandard methods is a triple $\langle\mathbb{V}(X),\mathbb{V}(Y),\ast\rangle$,
where 
\begin{enumerate}
\item $X$ is an infinite set; 
\item $\mathbb{V}(X),\mathbb{V}(Y)$ are the superstructures on $X,Y$ respectively; 
\item $\ast:\,\mathbb{V}(X)\rightarrow\mathbb{V}(Y)$ is a star map, namely
it satisfies the transfer principle. 
\end{enumerate}
We say that $\langle\mathbb{V}(X),\mathbb{V}(Y),\ast\rangle$ is a
single superstructure model of nonstandard methods when $X=Y$. 
\end{defn}
From now on, we will use only single superstructure models of nonstandard
methods. This is not restrictive: as proven in \cite[Section 3]{02}
(see also \cite{01}), every superstructure model is isomorphic to
a single superstructure one.

The existence of saturated single superstructure models of nonstandard
methods can be proven in different ways: we refer to \cite{03}, where
single superstructure models are constructed by means of the so-called
Alpha Theory, and to the nonstandard set theory \textit{$^{\ast}$\negmedspace{}}
ZFC introduced by M.~Di Nasso in \cite{07}, where the enlarging
map $\ast$ is defined for every set of the universe. Similar ideas
have been studied, in the context of iterated ultrapowers, by K.~Kunen
in \cite{16}, by K.~Hrbá\v{c}ek in \cite{14} and by K.~Hrbá\v{c}ek,
O.~Lessmann, R.~O'Donovan in \cite{15}. A clear presentation of
iterated ultrapowers can also be found in \cite[Section 6.5]{05}.

The main peculiarity of single superstructure models of nonstandard
methods is that they allow to iterate the $\ast$-map. This allows
to simplify certain proofs: for example, in \cite{09} the structure
$\prescript{\ast\ast}{}{\N}$, obtained by iterating twice the star
map applied to $\N$, is used to give a rather short proof of Ramsey
Theorem. 

Iterated hyperextensions have already been studied in previous publications
(e.g~\cite{09,10,17,18,19,20,21}). In this Section, we will recall
only the main definitions and properties that will be used in the
rest of the paper. 
\begin{defn}
We define by induction the family $\langle S_{n}\mid n\in\N\rangle$
of functions $S_{n}:\mathbb{V}(X)\rightarrow\mathbb{V}(X)$ by setting
$S_{0}=id$ and, for every $n\geq0$, $S_{n+1}=\ast\circ S_{n}$. 
\end{defn}
Let $Y$ be a set in $\mathbb{V}(X)$. Notice that $S_{2}\left(\prescript{\ast}{}{Y}\right):=\prescript{\ast\ast}{}{Y}$
is a nonstandard extension of both $Y$ and $\prescript{\ast}{}{Y}$.
Intuitively, this extension resembles the extension from $Y$ to $\prescript{\ast}{}{Y}$.
For example, if $Y=\mathbb{N}$, the fact that $\prescript{\ast}{}{\N}$
is an end extension of $\N$, namely that 
\[
\forall\eta\in\prescript{\ast}{}{\N}\setminus\N,\,\forall n\in\N\,\eta>n,
\]
can be transferred to 
\[
\forall\eta\in\prescript{\ast\ast}{}{\N}\setminus\prescript{\ast}{}{\N},\,\forall n\in\prescript{\ast}{}{\N}\,\eta>n,
\]
which is the formula expressing that $\prescript{\ast\ast}{}{\N}$
is an end extension of $\prescript{\ast}{}{\N}$.

However, not all the basic properties of the extension from $Y$ to
$\prescript{\ast}{}{Y}$ holds also for the extension $\prescript{\ast\ast}{}{Y}$
of $\prescript{\ast}{}{Y}$: for example, the fact that $\prescript{\ast}{}{A}=A$
for every finite subset of $\mathbb{N}$ (as usual, we identify every
number $n\in\mathbb{N}$ with $\prescript{\ast}{}{n}$) does not hold
true for $\prescript{\ast}{}{\N}$. Just observe that if $\alpha\in\prescript{\ast}{}{\N}\setminus\N$
then, by transfer, $\prescript{\ast}{}{\alpha}\in\prescript{\ast\ast}{}{\N}\setminus\prescript{\ast}{}{\N}$,
hence $\prescript{\ast}{}{\{\alpha\}}=\{\prescript{\ast}{}{\alpha}\}\neq\{\alpha\}$.

In any case, we have the following result, which is a trivial consequence
of the composition properties of elementary embeddings: 
\begin{thm}
For every positive natural number $n$, $\langle\mathbb{V}(X),\mathbb{V}(X),S_{n}\rangle$
is a single superstructure model of nonstandard methods. 
\end{thm}
In certain cases, as we will show in Section \ref{sec: tensor pairs},
it is helpful to consider the following extension of $\N$: 
\begin{defn}
\label{def: omega hyperextension}Let $\langle\mathbb{V}(X),\mathbb{V}(X),\ast\rangle$
be a superstructure model of nonstandard methods. We call $\omega$-hyperextension
of $X$, and denote by $\prescript{\bullet}{}{X}$, the union of all
hyperextensions $S_{n}(X)$: 
\[
\prescript{\bullet}{}{X}=\bigcup_{n\in\mathbb{N}}S_{n}(X).
\]
\end{defn}
Since 
\begin{center}
$\langle S_{n}(X)\mid n<\omega\rangle$ 
\par\end{center}

is an elementary chain of extensions, we have that $\prescript{\bullet}{}{X}$
is a hyperextension of $X$. Moreover, as $\prescript{\bullet}{}{A}\supset S_{n}(A)\supset A$
for every $A\subseteq X$, we have the following trivial result: 
\begin{prop}
\label{prop: enlarging properties}Let $n\in\mathbb{N}$ and let $\kappa$
be a cardinal number. Then the implications $(1)\Rightarrow(2)\Rightarrow(3)$
hold, where 
\begin{enumerate}
\item $\langle\mathbb{V}(X),\mathbb{V}(X),\ast\rangle$ has the $\kappa$-enlarging
property; 
\item $\langle\mathbb{V}(X),\mathbb{V}(X),S_{n}\rangle$ has the $\kappa$-enlarging
property; 
\item $\langle\mathbb{V}(X),\mathbb{V}(X),\bullet\rangle$ has the $\kappa$-enlarging
property. 
\end{enumerate}
\end{prop}
However, let us notice that the previous result does not hold, in
general, if we substitute enlarging with saturation. In fact, $\prescript{\bullet}{}{X}$
has\textbf{ }cofinality $\aleph_{0}$ (which is in contrast with $\kappa$-saturation
properties for $\kappa>\aleph_{0}$, as the cofinality is always at
least as great as the cardinal saturation), since a countable right
unbounded sequence in $\prescript{\bullet}{}{\N}$ can be constructed
by choosing, for every natural number $n$, an hypernatural number
$\alpha_{n}$ in $S_{n+1}(X)\setminus S_{n}(X)$. 
\begin{rem}
Of course, Definition \ref{def: omega hyperextension} could be easily
generalized by substituting $\omega$ with other ordinal numbers.
For example, the $\omega+1$ extension of $X$ is 
\[
\prescript{\ast}{}{\left(\,^{\bullet}\N\right)}=\prescript{\ast}{}{\left(\bigcup_{n\in\mathbb{N}}S_{n}(X)\right)}=\bigcup_{\eta\in\prescript{\ast}{}{\N}}S_{\eta}(X),
\]
where $S_{\eta}$ is the value in $\eta$ of the extension of the
map $S:\N\rightarrow\wp\left(\prescript{\bullet}{}{\N}\right)$. In
any case, in this paper we will not consider these more general instances
of iterated hyperextensions. 
\end{rem}

\section{\label{sec:Monads}Monads}

In the following, we will use the symbol $\star$ to denote a generic
nonstandard extensions (which could be $\ast,S_{n}\,\text{or}\,\bullet$),
reserving to $\ast$ and $\bullet$ the meanings given in Section
\ref{sec: Iterated-Hyperextensions}. We hope that this will increase
the readability of the paper\footnote{At least, we hope that this will not decrease the readability of the
paper.}.

Monads of filters were first introduced by W.A.J.~Luxembourg in \cite{22}.
In the past few years, monads of ultrafilters on $\N$ have been used
to prove many results in combinatorial number theory, especially in
the context of the partition regularity of equations (see e.g. \cite{09,10,11,12,17,18,19,20,21}).
However, it seems that to extend the range of applications of these
methods, a deeper study of monads in a wider generality is needed.
Our aim in this section is to start such a study. We will adopt the
framework of iterated nonstandard hyperextensions, since they provide
a simpler setting for the study of monads, as we are going to show. 
\begin{defn}
Let $Y$ be a set in $\mathbb{V}(X)$ and let $\langle\mathbb{V}(X),\mathbb{V}(X),\ast\rangle$
be a single superstructure model of nonstandard methods. Let $\U$
be an ultrafilter on $Y$. For every $n\in\N$ we let 
\[
\mu_{n}(\U):=\bigcap_{A\in\U}S_{n}(A)
\]
(with the agreement that $\mu(\U):=\mu_{1}(\N)$) and 
\[
\mu_{\infty}(\U):=\bigcap_{A\in\U}\prescript{\bullet}{}{A}.
\]
Elements of $\mu_{\infty}(\U)$ will be called generators of $\U$.
Finally, when we consider a generic extension $\langle\mathbb{V}(X),\mathbb{V}(X),\star\rangle$
we will write 
\[
\mu_{\star}(\U):=\bigcap_{A\in\U}\prescript{\star}{}{A}.
\]
\end{defn}
In general, monads can be empty if the extensions are not sufficiently
enlarged. However, we have the following result: 
\begin{thm}
\label{thm: bridge map}Let $Y$ be a set in $\mathbb{V}(X)$. Then
for every $\alpha\in\prescript{\star}{}{Y}$ the set 
\[
\mathcal{\mathfrak{U}}_{\alpha}:=\{A\subseteq Y\mid\alpha\in\prescript{\star}{}{Y}\}
\]
is a ultrafilter on $Y$. Moreover, if the extension $\star:Y\rightarrow\prescript{\star}{}{Y}$
has the $|\wp(Y)|^{+}$-enlarging property then $\mu_{\star}(\U)\neq\emptyset$
for every $\U\in\beta Y$. 
\end{thm}
\begin{proof}
That $\mathcal{U_{\alpha}}$ is a ultrafilter is straightforward.
The second claim follows as every ultrafilter $\U$ on $Y$ is a family
with the finite intersection property and cardinality $|\wp(Y)|$,
and the $|\wp(Y)|^{+}$-enlarging property hence entails that $\mu_{\star}(\U)\neq\emptyset$. 
\end{proof}
Monads can be used to identify every ultrafilter with the trace of
a principal one on a higher level: in fact, if $\alpha\in\prescript{\star}{}{Y}$
then $\alpha\in\mu(\U)$ if and only if $\U=tr_{Y}(P(\alpha))$, where
\[
P(\alpha):=\left\{ A\subseteq\prescript{\star}{}{Y}\mid\alpha\in A\right\} 
\]
is the principal ultrafilter generated by $\alpha$ on $\prescript{\star}{}{Y}$
and, for every ultrafilter $\V$ on $\prescript{\star}{}{Y}$, we
set 
\[
tr_{Y}\left(\V\right):=\left\{ A\subseteq Y\mid\prescript{\star}{}{Y}\in\V\right\} .
\]

\begin{defn}
For $\alpha,\beta\in\prescript{\star}{}{Y}$ we will write $\alpha\sim_{Y}\beta$
if the ultrafilters generated on $Y$ by $\alpha$ and $\beta$ coincide. 
\end{defn}
A remark is in order: in previous papers, the equivalence relation
$\sim_{Y}$ was denoted by $\sim_{\U}$ or $\underset{u}{\sim}$ (see
e.g.~\cite{09,10}). However, here we prefer to use the notation
$\sim_{Y}$ as this equivalence relation depends on the set on which
we are constructing the ultrafilters. To better explain what we mean,
let $\alpha\neq\beta\in\prescript{\ast}{}{\N}$ be such that $\alpha\sim_{\N}\beta$.
Then (as we will prove in Proposition \ref{basic properties monads})
$\prescript{\ast}{}{\alpha}\sim_{\N}\prescript{\ast}{}{\beta}$. However,
$\prescript{\ast}{}{\alpha}\nsim_{\prescript{\ast}{}{\N}}\prescript{\ast}{}{\beta}$!
In fact, the ultrafilter generated by $\prescript{\ast}{}{\alpha}$
on $\prescript{\ast}{}{\N}$ is 
\[
\left\{ A\subseteq\prescript{\ast}{}{\N}\mid\prescript{\ast}{}{\alpha}\in\prescript{\ast}{}{A}\right\} =\left\{ A\subseteq\prescript{\ast}{}{\N}\mid\alpha\in A\right\} =P(\alpha),
\]
and analogously the ultrafilter generated by $\prescript{\ast}{}{\beta}$
on $\prescript{\ast}{}{\N}$ is $P(\beta)$, and $P(\alpha)\neq P(\beta)$
since $\alpha\neq\beta$.

When we work in $\omega$-hyperextensions, it is useful to study the
relationships between sets of generators of the same ultrafilter in
different extensions. To do this, let us fix a definition: 
\begin{defn}
Let $\langle\mathbb{V}(X),\mathbb{V}(X),\ast\rangle$ be a single
superstructure model of nonstandard methods and let $Y$ be a set
in $\mathbb{V}(X)$. We say that $Y$ is coherent if $Y\subseteq\prescript{\ast}{}{Y}$.
We say that $Y$ is completely coherent if $A$ is coherent for every
$A\subseteq Y$. 
\end{defn}
Notice that if $Y$ is coherent then $S_{n}(Y)\subseteq S_{m}(Y)$
for every $n\leq m$. 
\begin{example}
$\N$ is completely coherent, as we identify every $n\in\N$ with
$\prescript{\ast}{}{n}$. However, if $\alpha\in\prescript{\ast}{}{\N}\setminus\N$,
then $\{\alpha\}$ is not coherent, since $\prescript{\ast}{}{\{\alpha\}}=\{\prescript{\ast}{}{\alpha}\}$.
Finally, $\N\cup\left\{ \alpha\right\} $ is coherent but not completely
coherent. 
\end{example}
\begin{thm}
Let $\langle\mathbb{V}(X),\mathbb{V}(X),\ast\rangle$ be a single
superstructure model of nonstandard methods and let $Y$ be a set
in $\mathbb{V}(X)$. The following are equivalent: 
\begin{enumerate}
\item $Y$ is completely coherent; 
\item for all $y\in Y$ $y=\prescript{\ast}{}{y}$. 
\end{enumerate}
\end{thm}
\begin{proof}
$(1)\Rightarrow(2)$ Let $y\in Y$. As $Y$ is completely coherent,
we have that $\{y\}\subseteq\prescript{\ast}{}{\{y\}}=\{\prescript{\ast}{}{y}\}$,
hence $y=\prescript{\ast}{}{y}$.

$(2)\Rightarrow(1)$ Let $A\subseteq Y$. Then $A\subseteq\prescript{\ast}{}{A}$
since, for every $a\in A$, we have that $a=\prescript{\ast}{}{a}\in\prescript{\ast}{}{A}$. 
\end{proof}
\begin{prop}
\label{basic properties monads}Let $Y$ be a set in $\mathbb{V}(X)$.
For every ultrafilter $\U$ on $Y$, for every $n\in\N$ the following
properties hold: 
\begin{enumerate}
\item \label{enu: monads 1}if $Y$ is completely coherent then $\mu_{n}(\U)\subseteq\mu_{n+1}(\U)$; 
\item \label{enu: monads 2}$\prescript{\ast}{}{\mu_{n}(\U)}\subseteq\mu_{n+1}(\U);$ 
\item \label{enu: monads 3}$\alpha\in\mu_{\infty}(\U)\Leftrightarrow\prescript{\ast}{}{\alpha}\in\mu_{\infty}(\U)$. 
\end{enumerate}
\end{prop}
\begin{proof}
(\ref{enu: monads 1}) Just notice that, for every $A\subseteq Y$,
since $A$ is coherent we have that $S_{n}(A)\subseteq S_{n+1}(A)$
for every $n\in\mathbb{N}$.

(\ref{enu: monads 2}) For every $A\in\U$ $\mu_{n}(\U)\subseteq S_{n}(A)$.
Hence, by transfer, $\prescript{\ast}{}{\mu_{n}(\U)}\subseteq\prescript{\ast}{}{S_{n}(A)}=S_{n+1}(A)$,
and so $\prescript{\ast}{}{\mu_{n}(\U)}\subseteq\bigcap_{A\in\U}S_{n+1}(A)=\mu_{n+1}(\U).$

(\ref{enu: monads 3}) $\alpha\in\mu_{\infty}(\U)\Leftrightarrow\exists n\in\N\,\alpha\in\mu_{n}(\U)\Leftrightarrow\exists n\in\N\,\prescript{\ast}{}{\alpha}\in\mu_{n+1}(\U)\Leftrightarrow\prescript{\ast}{}{\alpha}\in\mu_{\infty}(\U)$. 
\end{proof}
It is well-known that functions $f:Y_{1}\rightarrow Y_{2}$ can be
lifted to $\overline{f}:\beta Y_{1}\rightarrow\beta Y_{2}$ by setting
for every $\U\in\beta Y_{1}$ 
\[
\overline{f}(\U):=\left\{ A\subseteq Y_{2}\mid f^{-1}(A)\in\U\right\} .
\]
As one might expect, the monad of $\overline{f}(\U)$ can be expressed
in terms of the monad of $\U$, as shown in the following Theorem,
that generalized similar results proven by M.~Di Nasso\footnote{Notice that in \cite{10} these properties are proven in the case
$X=Y=\N$; however, these arguments used in these proofs work also
in the present case. } in \cite[Propositions 11.2.4,11.2.10 and Theorem 11.2.7]{10} in
the context of $\N$: 
\begin{thm}
\label{thm:functions and monads}Let $A,B\in\mathbb{V}(X)$ be sets,
let $f:\,A\rightarrow B$ and let $\U\in\beta A$. Then the following
facts hold: 
\begin{enumerate}
\item If $A=B$ and $\alpha,\prescript{\star}{}{f}(\alpha)\in\mu(\U)$ then
$\alpha=\prescript{\star}{}{f}(\alpha)$; 
\item $\mu(f(\U))=\prescript{\star}{}{f}(\mu(\U))$. 
\end{enumerate}
\end{thm}
\begin{proof}
The proof in the general case is identical to that given for $A=B=\mathbb{N}$
in $\prescript{\ast}{}{\N}$, see \cite[Propositions 11.2.4,11.2.10 and Theorem 11.2.7]{10}. 
\end{proof}

\section{\label{sec: tensor pairs}Arbitrary tensor products and pairs}

\subsection{Tensor products}

A key notion in ultrafilters theory is that of tensor product of ultrafilters: 
\begin{defn}
Let $S_{1},S_{2}$ be sets in $\mathbb{V}(X)$ and let $\U_{1}\in\beta S_{1},\,\U_{2}\in\beta S_{2}$.
The tensor product of $\U_{1}$ and $\U_{2}$ is the unique ultrafilter
on $S_{1}\times S_{2}$ such that for every $A\subseteq S_{1}\times S_{2}$
\[
A\in\U_{1}\otimes\U_{2}\Leftrightarrow\left\{ s_{1}\in S_{1}\mid\left\{ s_{2}\in S_{2}\mid\left(s_{1},s_{2}\right)\in A\right\} \in\U_{2}\right\} \in\U_{1}.
\]

Moreover, we set 
\[
\beta S_{1}\otimes\beta S_{2}=\left\{ \U_{1}\otimes\U_{2}\mid\U_{1}\in\beta S_{1},\U_{2}\in\beta S_{2}\right\} .
\]
\end{defn}
Tensor products are strictly related with the notion of double limits
along ultrafilters and Rudin-Keisler order (see \cite[Section 11.1]{13}
for the case $S_{1}=S_{2}=S$, $S$ discrete space). However, we will
not adopt this topological point of view here. For us, tensor products
are important because of the role they play in many applications,
especially in combinatorial number theory.
\begin{example}
Let $(S,\cdot)$ be a semigroup. Let $f:S^{2}\rightarrow S$ be $f(a,b)=a\cdot b$.
Then $f\left(\U,\V\right)=\overline{f}\left(\U\otimes\V\right)=\U\odot\V$
for every $\U,\V\in\beta S$. 
\end{example}
\begin{example}
Let $\mathcal{F}=\N^{\N}$ and let $H:\N\times\mathcal{F\rightarrow\N}$
be the function $H(n,f)=f(n)$. Then\footnote{These kind of ultrafilters are important when studying combinatorial
properties of $\N$ by means of the so-called $\mathcal{F}$-finite
embeddabilities, see \cite{21}.} 
\[
H\left(\U,\V\right)=\U\otimes_{\mathcal{F}}\V:=\left\{ A\subseteq\N\mid\left\{ n\in\N\mid\left\{ f\in\mathcal{F}\mid f(n)\in A\right\} \in\V\right\} \in\U\right\} .
\]
\end{example}
The first trivial observation about tensor products is the following: 
\begin{lem}
If $S_{1}$ or $S_{2}$ is finite then $\beta\left(S_{1}\otimes S_{2}\right)=\beta S_{1}\otimes\beta S_{2}$. 
\end{lem}
\begin{proof}
Let us prove the case with $S_{1}$ finite, as the other case is similar.
Let $S_{1}=\left\{ a_{1},\dots,a_{n}\right\} $. Let $\U\in\beta\left(S_{1}\otimes S_{2}\right)$.
For $i=1,\dots,n$ let $A_{i}=\left\{ \left(a_{i},s_{2}\right)\mid s_{2}\in S_{2}\right\} $.
As $S_{1}\times S_{2}=\bigcup_{i\leq n}A_{i}$, there exists a unique
$i\leq n$ such that $A_{i}\in\U$. Let $\mathcal{\mathfrak{U}}_{a_{i}}\in\beta S_{1}$
be the principal ultrafilter on $a_{i}$ and let 
\[
\U_{2}=\overline{\pi_{2}}(\U):=\left\{ A\subseteq S_{2}\mid\left\{ \left(a_{i},s\right)\mid s\in A\right\} \in\U\right\} \in\beta S_{2}.
\]
Then, by construction, we have that for every $A\subseteq S_{1}\times S_{2}$
\[
A\in\mathcal{\mathfrak{U}}_{a_{i}}\otimes\U_{2}\Leftrightarrow\left\{ s_{1}\in S_{1}\mid\left\{ s_{2}\in S_{2}\mid\left(s_{1},s_{2}\right)\in A\right\} \in\U_{2}\right\} \in\mathcal{\mathfrak{U}}_{a_{i}}\Leftrightarrow
\]
\[
\left\{ s_{2}\in S_{2}\mid\left(a_{i},s_{2}\right)\in A\right\} \in\U_{2}\Leftrightarrow A\in\U,
\]
hence $\U=\mathcal{\mathfrak{U}}_{a_{i}}\otimes\U_{2}\in\beta S_{1}\otimes\beta S_{2}$. 
\end{proof}
To develop a deeper study of tensor products, our goal in this section
is to give several characterizations of monads of tensor products
of ultrafilters. The first question that we want to answer is: what
is the relationship between $\mu(\U)\times\mu(\V)$ and $\mu(\U\otimes\V)$?
Let us start with a definition: 
\begin{defn}
Let $\U_{1}\in\beta S_{1},\U_{2}\in\beta S_{2}$. We denote by $\mathcal{F}\left(\U_{1},\U_{2}\right)$
the filter on $S_{1}\times S_{2}$ given by 
\[
\mathcal{F}\left(\U_{1}\times\U_{2}\right)=\left\{ B\in S_{1}\times S_{2}\mid\exists A_{1}\in\U_{1},A_{2}\in\U_{2}\,\text{s.t.}\,A_{1}\times A_{2}\subseteq B\right\} .
\]
\end{defn}
In general, $\mathcal{F}\left(\U_{1}\times\U_{2}\right)$ is just
a filter and not an ultrafilter; its relationship with $\mu\left(\U_{1}\right)$,
$\mu\left(\U_{2}\right)$ and $\mu\left(\U_{1}\otimes\U_{2}\right)$
is clarified in the following Proposition. 
\begin{prop}
\label{prop: product monads}Let $\U_{1}\in\beta S_{1}$ and $\U_{2}\in\beta S_{2}$.
Then 
\begin{equation}
\mu(\U_{1})\times\mu(\U_{2})=\bigcup_{\mathcal{W}\in U\left(\U_{1}\times\U_{2}\right)}\mu(\mathcal{W})\supseteq\mu(\U_{1}\otimes\U_{2}),\label{eq: monads}
\end{equation}
where $U\left(\U_{1}\times\U_{2}\right)=\left\{ \mathcal{W}\in\beta\left(S_{1}\times S_{2}\right)\mid\mathcal{W\supseteq\mathcal{F}}\left(\U_{1}\times\U_{2}\right)\right\} $. 
\end{prop}
\begin{proof}
First of all, let us notice that $\U_{1}\otimes\U_{2}\in U\left(\U_{1}\times\U_{2}\right)$,
as clearly $A\times B\in\U_{1}\otimes\U_{2}\,\forall A\in\U_{1},B\in\U_{2}$.
Therefore we are left to prove that $\mu(\U_{1})\times\mu(\U_{2})=\bigcup_{\mathcal{W}\in U\left(\U_{1}\times\U_{2}\right)}\mu(\mathcal{W})$.

$\subseteq$: Let $\alpha\in\mu\left(\U_{1}\right),\beta\in\mu\left(\U_{2}\right)$.
Let $\W=\mathcal{\mathfrak{U}}_{(\alpha,\beta)}.$ Then $\W\in U\left(\U_{1}\times\U_{2}\right)$
as, for every $A\in\U_{1},B\in\U_{2}$ $(\alpha,\beta)\in\prescript{\star}{}{A}\times\prescript{\star}{}{B}=\prescript{\star}{}{\left(A\times B\right)}.$

$\supseteq$: Let $\W\in U\left(\U_{1}\times\U_{2}\right)$. Let $(\alpha,\beta)\in\mu(\W)$.
For every $A\in\U_{1},B\in\U_{2}$ $A\times B\in\W$, hence 
\[
(\alpha,\beta)\in\bigcap_{A\in\U_{1},B\in\U_{2}}\prescript{\star}{}{\left(A\times B\right)}=\mu(\U_{1})\times\mu(\U_{2}).\qedhere
\]
\end{proof}
\begin{cor}
For every $\alpha\in\prescript{\star}{}{S_{1}},\beta\in\prescript{\star}{}{S_{2}}$
$\mathfrak{U_{(\alpha,\beta)}\supseteq}\mathcal{F}\left(\mathfrak{U_{\alpha}\times\mathfrak{U}_{\beta}}\right)\Leftrightarrow\alpha\sim_{S_{1}}\gamma,\beta\sim_{S_{2}}\delta$. 
\end{cor}
In particular, as a consequence of Proposition \ref{prop: product monads}
we have that the map $\otimes:\,\beta S_{1}\times\beta S_{2}\rightarrow\beta\left(S_{1}\times S_{2}\right)$
is injective but not surjective, in general. Moreover, as it is known,
this entails that $\mu\left(\U_{1}\otimes\U_{2}\right)=\mu\left(\U_{1}\right)\times\mu\left(\U_{2}\right)$
if and only if $\mathcal{F}\left(\U_{1}\times\U_{2}\right)=\U_{1}\otimes\U_{2}$
(see also \cite[Chapter 1]{key-1}), 

To characterize when such a situation happens, let us recall the following
definitions: 
\begin{defn}
Let $\U\in\beta S$ and let $\kappa$ be a cardinal number. The norm
of $\U$ is the cardinal 
\[
\left\Vert \U\right\Vert =\min_{A\in\U}\left|A\right|.
\]
Moreover, $\U$ is $\kappa^{+}$-complete if for every family $\left\{ A_{i}\mid i\in I\right\} \subseteq\U$
with cardinality $\left|I\right|<\kappa^{+}$ we have $\bigcap_{i\in I}A_{i}\in\U$. 
\end{defn}
The problem of characterizing ultrafilters $\U_{1},\U_{2}$ such that
$\U_{1}\otimes\U_{2}=\mathcal{F}\left(\U_{1}\times\U_{2}\right)$
was already considered, and solved, by A.~Blass in \cite[Section 4]{key-1}.
We recall (and reprove for completeness) his characterization in the
following Theorem:
\begin{thm}
\label{thm: chara tensor products}Let $\U_{1}\in\beta S_{1},\U_{2}\in\beta S_{2}$.
The following facts are equivalent: 
\begin{enumerate}
\item $\U_{1}\otimes\U_{2}=\mathcal{F}\left(\U_{1}\times\U_{2}\right);$ 
\item $\forall A\in\U_{1}\,\forall\left\{ B_{i}\mid i\in A\right\} \subseteq\U_{2}\,\exists C\in\U_{1}\,\bigcap_{c\in C\cap A}B_{c}\in\U_{2}.$ 
\end{enumerate}
\end{thm}
\begin{proof}
$(1)\Rightarrow(2)$ Let $A\in\U_{1}$ and let $\left\{ B_{i}\mid i\in A\right\} \subseteq\U_{2}$.
Let $S\subseteq S_{1}\times S_{2}$ be the set 
\[
S=\bigcup_{i\in A}\left\{ i\right\} \times B_{i}.
\]
As $A\in\U_{1}$, by definition of tensor product we have that $S\in\U_{1}\otimes\U_{2}$.
But then, as $\U_{1}\otimes\U_{2}=\mathcal{F}\left(\U_{1}\times\U_{2}\right),$
there exist $D_{1}\in\U_{1},D_{2}\in\U_{2}$ such that $D_{1}\times D_{2}\subseteq S$.
Hence for every $c\in C:=A\cap D_{1}$ we have that $D_{2}\subseteq B_{c}$,
so in particular 
\[
D_{2}\subseteq\bigcap_{c\in C}B_{c},
\]

hence $\bigcap_{c\in C}B_{c}\in\U_{2}$.

$(2)\Rightarrow(1)$ Let $S\in\U_{1}\otimes\U_{2}$. By definition,
\[
A:=\left\{ s_{1}\in S_{1}\mid\left\{ s_{2}\in S_{2}\mid\left(s_{1},s_{2}\right)\in S\right\} \in\U_{2}\right\} \in\U_{1}.
\]
For every $i\in A$ let $B_{i}=\left\{ s_{2}\in S_{2}\mid\left(i,s_{2}\right)\in S\right\} \in\U_{2}$.
Then $\left\{ B_{i}\mid i\in A\right\} \subseteq\U_{2}$, hence by
hypothesis there exists $C\in\U_{1}$ so that $B:=\bigcap_{c\in C\cap A}B_{c}\in\U_{2}$.
But then by construction $(C\cap A)\times B\subseteq S$, and so $S\in\mathcal{F}\left(\U_{1}\times\U_{2}\right)$. 
\end{proof}
\begin{example}
Let $S_{1}=S_{2}=\N$. Let $\U_{1},\U_{2}\in\beta\N$. Then the following
are equivalent\footnote{This is a well-known fact: see e.g. \cite[Remark 11.5.5]{10}.}: 
\begin{enumerate}
\item $\U_{1}\otimes\U_{2}=\mathcal{F}\left(\U_{1}\times\U_{2}\right)$; 
\item $\exists i\in\left\{ 0,1\right\} $ such that $\U_{i}$ is principal. 
\end{enumerate}
In fact, that $(2)\Rightarrow(1)$ is straightforward. On the other
hand, let us assume $(1)$, and assume that $\U_{1}$ and $\U_{2}$
are not principal. Let $A$ be any set in $\U$ and, for every $a\in A$,
let $B_{a}=\left\{ n\in\N\mid n>a\right\} \in\U_{2}$. By Theorem
\ref{thm: chara tensor products} there exists $C\in\U_{1}$ such
that $\bigcap_{a\in A\cap C}B_{a}\in\U_{2}$. And this cannot be,
as $A\cap C$ is infinite (since $\U_{1}$ is not principal) and hence
$\bigcap_{a\in A\cap C}B_{a}=\emptyset.$ 
\end{example}
We want to generalize the previous example and solve the following
two problems:
\begin{enumerate}
\item For which $\U_{1}\in\beta S_{1}$ we have that $\forall\U_{2}\in\beta S_{2}$
$\U_{1}\otimes\U_{2}=\mathcal{F}\left(\U_{1}\times\U_{2}\right)$?
\item For which $\U_{2}\in\beta S_{2}$ we have that $\forall\U_{1}\in\beta S_{1}$
$\U_{1}\otimes\U_{2}=\mathcal{F}\left(\U_{1}\times\U_{2}\right)$?
\end{enumerate}
Although it is not evident from Theorem \ref{thm: chara tensor products},
the property $\U_{1}\otimes\U_{2}=\mathcal{F}\left(\U_{1}\times\U_{2}\right)$
is symmetic in $\U_{1},\U_{2}$, in the sense that 
\[
\U_{1}\otimes\U_{2}=\mathcal{F}\left(\U_{1}\times\U_{2}\right)\Leftrightarrow\U_{2}\otimes\U_{1}=\mathcal{F}\left(\U_{2}\times\U_{1}\right)
\]
(this basic observations was pointed out to us by A.~Blass, see also
\cite[Corollary 9, Section 4]{key-1}). Therefore problems 1 and 2
are equivalent: a solution of the first entails directly a solution
of the second. And the second problem is rather simple to solve:
\begin{thm}
Let $\U_{2}\in\beta S_{2}$ and let $\kappa=|S_{1}|$. The following
are equivalent: 
\begin{enumerate}
\item $\forall\U_{1}\in\beta S_{1}\,\U_{1}\otimes\U_{2}=\mathcal{F}\left(\U_{1}\times\U_{2}\right);$ 
\item $\U_{2}$ is $\kappa^{+}$-complete. 
\end{enumerate}
\end{thm}
\begin{proof}
$(1)\Rightarrow(2)$ Without loss of generality, we can assume that
$S_{1}=\kappa$. By contrast, let us suppose that $\U_{2}$ is not
$\kappa^{+}$-complete. Let 
\[
\lambda=\min\left\{ \mu\mid\exists F\subseteq\mathcal{U}\,\text{with}\,|F|=\mu\,\text{and\,}\bigcap_{B\in F}B\notin\U_{2}\right\} .
\]
As $\U_{2}$ is not $\kappa^{+}$-complete, $\lambda\leq\kappa$.
Let $\left\{ B_{i}\mid i<\lambda\right\} \subseteq\U_{2}$ be such
that $\bigcap_{i<\lambda}B_{i}\notin\U_{2}$. For every $i<\lambda$
we set $D_{i}=\bigcap_{j\leq i}B_{j}$. By the definition of $\lambda$
we have that every $D_{i}\in\U_{2}$, as it is an intersection of
fewer than $\lambda$ elements of $\U_{2}$, and clearly $D_{i}\supseteq D_{j}$
for every $i\leq j$. Now let $\U_{1}\in\beta\lambda\subseteq\beta\kappa$
be an ultrafilter that extends the filter of co-initial sets on $\lambda$,
so that every set $C\in\U_{1}$ is cofinal in $\lambda$. By hypothesis,
$\U_{1}\otimes\U_{2}=\mathcal{F}\left(\U_{1}\times\U_{2}\right)$.
If we set $A=\lambda$, by Theorem \ref{thm: chara tensor products}
we deduce that $\exists C\in\U_{1}\,\bigcap_{c\in C}D_{c}\in\U_{2}.$
But $C$ is cofinal in $\lambda$ and $\left\{ D_{i}\right\} _{i<\lambda}$
is a decreasing sequence, so $\bigcap_{c\in C}D_{c}=\bigcap_{i<\lambda}D_{i}=\bigcap_{i<\lambda}B_{i}\notin\U_{2}$,
which is absurd.

$(2)\Rightarrow(1)$ Let $A\in\U_{1}$ and let $\left\{ B_{i}\mid i\in A\right\} \subseteq\U_{2}$.
Then, as $|A|\leq\kappa$, by $\kappa^{+}$-completeness $\bigcap_{i\in A}B_{i}\in\U_{2}$,
hence the condition of Theorem \ref{thm: chara tensor products} is
fulfilled by setting $C=A$. 
\end{proof}
Let us call a ultrafilter $\U_{2}\in\beta S_{2}$ such that $\forall\U_{1}\in\beta S_{1}\,\U_{1}\otimes\U_{2}=\mathcal{F}\left(\U_{1}\times\U_{2}\right)$
a factorizing ultrafilter. If $\lambda=\left|S_{2}\right|$, from
the previous Theorem we deduce that:
\begin{itemize}
\item if $\lambda\leq\kappa$ then the unique factorizing ultrafilters are
the principal ones;
\item if $\lambda>\kappa$ then nonprincipal factorizing ultrafilter $\U_{2}\in\beta S_{2}$
might or might not exist: for example, if $\lambda=\kappa^{+}$then
such a nonprincipal factorizing ultrafilter exists if and only if
$\kappa^{+}$ is measurable (the existence of such ultrafilters is
consistent with ZF but not with ZFC).
\end{itemize}

\subsection{Tensor pairs}

We now want to characterize tensor products in terms of their monads.
To do so, we introduce the following definition: 
\begin{defn}
Let $\left(\alpha,\beta\right)\in\prescript{\star}{}{\left(S_{1}\times S_{2}\right)}$.
We say that $(\alpha,\beta)$ is a tensor pair (notation: $\left(\alpha,\beta\right)$
TT) if $\mathfrak{U}_{(\alpha,\beta)}=\mathfrak{U}_{\alpha}\otimes\mathfrak{U}_{\beta}$. 
\end{defn}
As, in general, $\beta S_{1}\otimes\beta S_{2}\subsetneq\beta\left(S_{1}\times S_{2}\right)$,
not all pairs $\left(\alpha,\beta\right)\in\prescript{\star}{}{\left(S_{1}\times S_{2}\right)}$
are tensor pairs. When $S_{1}=S_{2}=\N$, many properties of tensor
pairs have been proven (in the context of non-iterated hyperextensions)
by M.\ Di Nasso in \cite{10} (see also \cite{17}). We plan to show
that most of these characterizations can be extended (sometimes in
an even more general form) to arbitrary tensor pairs, with some simplifications
given by the possibility of iterating the star map.

The main advantage when working in iterated hyperextensions is that
they allow to write down easily generators of tensor products: 
\begin{thm}
\label{thm:Tensor types in iterated hyperextensions}Let $n,m\in\mathbb{N}$,
let $S_{1},S_{2}\in\mathbb{V}(X)$ be sets with $S_{1}$ completely
coherent, let $\U\in\beta S_{1}$, $\V\in\beta S_{2}$ and let $\alpha\in\mu_{n}(\U),\,\beta\in\mu_{m}(\V)$.
Then $\left(\alpha,S_{n}(\beta)\right)\in\mu_{n+m}\left(\U\otimes\V\right)$. 
\end{thm}
\begin{proof}
Let $A\subseteq S_{1}\times S_{2}$. Then 
\[
A\in\mathfrak{\mathfrak{U}}_{\alpha}\otimes\mathfrak{U}_{\beta}\Leftrightarrow\left\{ s_{1}\in S_{1}\mid\left\{ s_{2}\in S_{2}\mid\left(s_{1},s_{2}\right)\in A\right\} \in\mathfrak{U}_{\beta}\right\} \in\mathfrak{U}_{\alpha}.
\]
Now, by definition, $\left\{ s_{2}\in S_{2}\mid\left(s_{1},s_{2}\right)\in A\right\} \in\mathfrak{U}_{\beta}$
if and only if 
\[
\beta\in S_{m}\left(\left\{ s_{2}\in S_{2}\mid\left(s_{1},s_{2}\right)\in A\right\} \right)=\left\{ s_{2}\in S_{m}\left(S_{2}\right)\mid\left(s_{1},s_{2}\right)\in S_{m}\left(A\right)\right\} ,
\]
as $S_{m}\left(s_{1}\right)=s_{1}$ for every $s_{1}\in S_{1}$, since
$S_{1}$ is completely coherent. But 
\[
\beta\in\left\{ s_{2}\in S_{m}\left(S_{2}\right)\mid\left(s_{1},s_{2}\right)\in S_{m}\left(A\right)\right\} \Leftrightarrow
\]
\[
\left\{ s_{1}\in S_{1}\mid\left(s_{1},\beta\right)\in S_{m}(A)\right\} \in\mathfrak{U}_{\alpha}\Leftrightarrow\left(\alpha,S_{n}(\beta)\right)\in S_{n+m}(A).\qedhere
\]
\end{proof}
\begin{cor}
\label{cor:operations semigroups}Let $\left(S,\cdot\right)\in\mathbb{V}(X)$
be a semigroup. Assume that $S$ is completely coherent. Let $\U,\V\in\beta S$,
let $\alpha\in\mu_{n}(\U),\,\beta\in\mu_{m}(\V)$. Then $\alpha\cdot S_{n}\left(\beta\right)\in\mu_{n+m}\left(\U\odot\V\right)$. 
\end{cor}
\begin{proof}
$\U\odot\V=\overline{f}\left(\U\otimes\V\right)$, where $f:S^{2}\rightarrow S$
is the function that maps every pair $(a,b)\in S^{2}$ in $a\cdot b$.
Hence we can conclude by applying Theorem \ref{thm:functions and monads}.(2)
as, by Theorem \ref{thm:Tensor types in iterated hyperextensions},
$\left(\alpha,S_{n}(\beta)\right)\in\mu_{n+m}\left(\U\otimes\V\right)$. 
\end{proof}
\begin{rem}
In Theorem \ref{thm: bridge map} we showed that $\bigslant{\prescript{\bullet}{}{Y}}{\sim_{Y}}\cong\beta Y$.
Theorem \ref{thm:Tensor types in iterated hyperextensions} allows
to refine this result when $Y=S$ is a semigroup: if we let $\odot:\prescript{\bullet}{}{S^{2}}\rightarrow\prescript{\bullet}{}{S}$
be the map such that, for every $\alpha,\beta\in\prescript{\bullet}{}{S}$
\[
\alpha\odot\beta=\alpha\cdot S_{h(a)}(\beta),
\]
where $h(\alpha)=\min\left\{ n\in\N\mid\alpha\in S_{n}(S)\right\} $,
we get that $\left(\beta S,\odot\right)$ and $\bigslant{\left(\prescript{\bullet}{}{S},\odot\right)}{\sim_{Y}}$
are isomorphic as semigroups\footnote{This result could be improved to a topological isomorphism by introducing
the star topology on $\prescript{\bullet}{}{S}$, but we will not
consider this topological approach here.}.
\end{rem}
To simplify the notations, from now on we will assume that $\left(\alpha,\beta\right)\in\prescript{\ast}{}{\left(S_{1}\times S_{2}\right)}$,
as the characterization for the general cases where $\alpha\in S_{n}\left(S_{1}\right),$
$\beta\in S_{m}\left(S_{2}\right)$ can be analogously deduced from
Theorem \ref{thm:Tensor types in iterated hyperextensions}.

In the case of non-iterated hyperextensions, tensor pairs have been
studied mostly for the product $\N\times\N$. In this case, a characterization
was given by C.\ Puritz in \cite{25}, where he proved that

\[
(\alpha,\beta)\ \text{TT}\Leftrightarrow\alpha<er(\beta),
\]

where

\[
er(\beta)=\left\{ \prescript{\ast}{}{f(\beta)}\mid f:\N\rightarrow\N,\prescript{\ast}{}{f(\beta)}\in\prescript{\ast}{}{\N}\setminus\N.\right\} 
\]

\begin{rem}
In $\prescript{\ast\ast}{}{\N}$ it is very simple to see that Puritz's
characterization is not symmetric, in the sense that the condition
$\beta>er(\alpha)$ does not entail that $(\alpha,\beta)$ is TT.
In fact, let $\alpha$ be a prime number in $\prescript{\ast}{}{\N}$
and let $\beta=\left(\prescript{\ast}{}{\alpha}\right)^{\alpha}$.
Then $\beta>er(\alpha)$, as $\left(\prescript{\ast\ast}{}{f}\right)(\alpha)=\left(\prescript{\ast}{}{f}\right)(\alpha)\in\prescript{\ast}{}{\N}$
for every $f\in\N^{\N}$, whilst $\beta\in\prescript{\ast\ast}{}{\N}\setminus\prescript{\ast}{}{\N}$.
However, if $f$ is the function such that, if $n=p_{1}^{a_{1}}\cdot\dots\cdot p_{h}^{a_{h}}\in\N$
is the factorization of $n$ as product of distinct prime numbers,
then 
\[
f(n)=\max_{j=1,\dots,h}a_{j},
\]
we have that $\left(\prescript{\ast\ast}{}{f}\right)(\beta)=\alpha$,
hence by Puritz's characterization $(\alpha,\beta)$ is not TT. 
\end{rem}
The main problem in extending Puritz's characterization to arbitrary
sets is that it uses the order relation on $\N$, whilst arbitrary
products of sets might not be ordered. However, several of the equivalent
characterization of Puritz's condition given by M.~Di Nasso in \cite{10},
as well as some new ones, hold true also for arbitrary tensor pairs: 
\begin{thm}
\label{thm: main TT theo}Let $S_{1},S_{2}\in\mathbb{V}(X)$ be sets
and let $\left(\alpha_{1},\alpha_{2}\right)\in\prescript{\ast}{}{\ensuremath{\left(S_{1}\times S_{2}\right)}}$.
The following are equivalent: 
\begin{enumerate}
\item \label{enu: Main TT Theo (1)}$\left(\alpha_{1},\alpha_{2}\right)$
is a tensor pair; 
\item \label{enu: Main TT Theo (2)}$\left(\alpha_{1},\alpha_{2}\right)\sim_{u}\left(\alpha_{1},\prescript{\ast}{}{\alpha_{2}}\right)$; 
\item \label{enu: Main TT Theo (3)}$\forall A\subseteq S_{1}\times S_{2}$
if $\left(\alpha_{1},\alpha_{2}\right)\in\prescript{\ast}{}{A}$ then
$\exists s\in S_{1}$ such that $\left(s,\alpha_{2}\right)\in\prescript{\ast}{}{A};$ 
\item \label{enu: Main TT Theo (4)}$\forall A\subseteq S_{1}\times S_{2}$
if $\left(s,\alpha_{2}\right)\in\prescript{\ast}{}{A}$ $\forall s\in S_{1}$
then $\left(\alpha_{1},\alpha_{2}\right)\in\prescript{\ast}{}{A};$ 
\item \label{enu: Main TT Theo (5)}for every sets $S_{3},S_{4}$, for every
functions $f:\,S_{1}\rightarrow S_{3}$, $g:\,S_{2}\rightarrow S_{4}$
$\left(\prescript{\ast}{}{f\left(\alpha_{1}\right)},\prescript{\ast}{}{g\left(\alpha_{2}\right)}\right)$
is a tensor pair; 
\item \label{enu: Main TT Theo (6)}there exist sets $S_{3},S_{4}$, a bijection
$f:\,S_{1}\rightarrow S_{3}$ and an injective function $g:\,S_{2}\rightarrow S_{4}$
such that $\left(\prescript{\ast}{}{f\left(\alpha_{1}\right)},\prescript{\ast}{}{g\left(\alpha_{2}\right)}\right)$
is a tensor pair. 
\end{enumerate}
\end{thm}
\begin{proof}
(\ref{enu: Main TT Theo (1)})$\Rightarrow$(\ref{enu: Main TT Theo (2)})
This is an immediate consequence of Theorem \ref{thm:Tensor types in iterated hyperextensions}.

(\ref{enu: Main TT Theo (2)})$\Rightarrow$(\ref{enu: Main TT Theo (3)})
Let $A\subseteq S_{1}\times S_{2}$ be such that $\left(\alpha_{1},\alpha_{2}\right)\in\prescript{\ast}{}{A}$.
As $\left(\alpha_{1},\alpha_{2}\right)\sim_{S_{1}\times S_{2}}\left(\alpha_{1},\prescript{\ast}{}{\alpha_{2}}\right)$,
we have that $\left(\alpha_{1},\prescript{\ast}{}{\alpha_{2}}\right)\in\prescript{\ast\ast}{}{A}$,
namely 
\[
\alpha_{1}\in\prescript{\ast}{}{\left\{ s\in S_{1}\mid\left(s,\alpha_{2}\right)\in\prescript{\ast}{}{A}\right\} },
\]
hence $\left\{ s\in S_{1}\mid\left(s,\alpha_{2}\right)\in\prescript{\ast}{}{A}\right\} \neq\emptyset$.

(\ref{enu: Main TT Theo (3)})$\Rightarrow$(\ref{enu: Main TT Theo (4)})
This is straightforward, as (\ref{enu: Main TT Theo (4)}) is the
contrapositive of (\ref{enu: Main TT Theo (3)}) applied to $A^{c}$.

(\ref{enu: Main TT Theo (4)})$\Rightarrow$(\ref{enu: Main TT Theo (5)})
By contrast: assume that there exists sets $S_{3},S_{4}$ and functions
$f:S_{1},\rightarrow S_{3},\,g:S_{2}\rightarrow S_{4}$ such that
$\left(\prescript{\ast}{}{f\left(\alpha_{1}\right)},\prescript{\ast}{}{g\left(\alpha_{2}\right)}\right)$
is not a tensor pair. Let $A\subseteq S_{3}\times S_{4}$ be such
that $\left(s_{3},\prescript{\ast}{}{g\left(\alpha_{2}\right)}\right)\in\prescript{\ast}{}{A}$
for every $s_{3}\in S_{3}$ but $\left(\prescript{\ast}{}{f\left(\alpha_{1}\right)},\prescript{\ast}{}{g\left(\alpha_{2}\right)}\right)\notin\prescript{\ast}{}{A}$.
Let 
\[
X_{A}=\left\{ \left(s_{1},s_{2}\right)\in S_{1}\times S_{2}\mid\left(f\left(s_{1}\right),g\left(s_{2}\right)\right)\in A\right\} .
\]
By construction, $\left(s_{1},\alpha_{2}\right)\in\prescript{\ast}{}{X_{A}}$
for every $s_{1}\in S_{1}$, hence by (\ref{enu: Main TT Theo (4)})
$\alpha_{1}\in\prescript{\ast}{}{X_{A}}$, namely $\left(\prescript{\ast}{}{f\left(\alpha_{1}\right)},\prescript{\ast}{}{g\left(\alpha_{2}\right)}\right)\in\prescript{\ast}{}{A}$,
which is absurd.

(\ref{enu: Main TT Theo (5)})$\Rightarrow$(\ref{enu: Main TT Theo (6)})
This is straightforward.

(\ref{enu: Main TT Theo (5)})$\Rightarrow$(\ref{enu: Main TT Theo (1)})
Just set $S_{1}=S_{3}$, $S_{2}=S_{4}$, $f=id_{S_{1}}$, $g=id_{S_{2}}$.

(\ref{enu: Main TT Theo (6)})$\Rightarrow$(\ref{enu: Main TT Theo (1)})
By contrast, assume that $\left(\alpha_{1},\alpha_{2}\right)$ is
not a tensor pair. By the equivalence (\ref{enu: Main TT Theo (5)})$\Leftrightarrow$(\ref{enu: Main TT Theo (1)}),
there exists $A\subseteq S_{1}\times S_{3}$ such that $\left(s,\alpha_{2}\right)\in\prescript{\ast}{}{A}$
$\forall s\in S_{1}$ but $\left(\alpha_{1},\alpha_{2}\right)\notin\prescript{\ast}{}{A}$.
Let 
\[
X_{A}=(f,g)(A)=\left\{ \left(s_{3},s_{4}\right)\mid\exists\left(s_{1},s_{2}\right)\in A\,f\left(s_{1}\right)=s_{3},g\left(s_{2}\right)=s_{4}\right\} .
\]
As $f$ is surjective and $\left(s,\alpha_{2}\right)\in\prescript{\ast}{}{A}$
$\forall s\in S_{1}$, we have that $\left(s_{3},\prescript{\ast}{}{g\left(\alpha_{2}\right)}\right)\in\prescript{\ast}{}{X_{A}}$
for every $s_{3}\in S_{3}$. Hence $\left(\prescript{\ast}{}{f\left(\alpha_{1}\right)},\prescript{\ast}{}{g\left(\alpha_{2}\right)}\right)\in\prescript{\ast}{}{X_{A}}$,
as $\left(\prescript{\ast}{}{f\left(\alpha_{1}\right)},\prescript{\ast}{}{g\left(\alpha_{2}\right)}\right)$
is a tensor pair. But as $X=(f,g)(A)$ and $f,g$ are 1-1, we deduce
that $\left(\alpha_{1},\alpha_{2}\right)\in\prescript{\ast}{}{A}$,
which is absurd. 
\end{proof}
To prove that Theorem \ref{thm: main TT theo} entails Puritz result
and allows for a simple characterization of tensor pairs in many cases,
let us introduce the following definition: 
\begin{defn}
Let $S_{1},S_{2}$ be given sets. Let $Y\subseteq\beta\left(S_{1}\times S_{2}\right)$.
We say that $(\alpha,\beta)\in\prescript{\ast}{}{\left(S_{1}\times S_{2}\right)}$
is a $Y$-tensor pair if it is a tensor pair and $\mathfrak{U}_{\left(\alpha,\beta\right)}\in Y$.
\end{defn}
The basic observation is the following: 
\begin{rem}
If $(\alpha,\beta)$ is a $Y$-tensor pair then $(\alpha,\beta)\in\prescript{\ast}{}{A}$
for every $A\in\bigcap_{\U\in Y}\U$ (i.e., $(\alpha,\beta)$ generates
the filter $\bigcap_{\U\in Y}\U$). 
\end{rem}
\begin{example}
If $S_{1}=S_{2}=\N$ and $Y=\left\{ \U\otimes\V\mid\U,\V\in\beta\N\setminus\N\right\} $
then $Y$-tensor pairs are tensor pairs with both entries infinite
and, as 
\[
\Delta=\left\{ (n,m)\mid n<m\right\} \in\mathcal{W}
\]
for every $\mathcal{W\in}Y$, this shows that $\left(\alpha,\beta\right)\in\prescript{\ast}{}{\Delta}$
for every tensor pair $\left(\alpha,\beta\right)$. But then, by applying
Theorem \ref{thm: main TT theo}.(\ref{enu: Main TT Theo (5)}) with
$S_{3}=S_{4}=\N$, we deduce that for every $f,g:\,\mathbb{N}\rightarrow\mathbb{N}$
with $\prescript{\ast}{}{f\left(\alpha_{1}\right)},\prescript{\ast}{}{g\left(\alpha_{2}\right)}$
infinite we have $\left(\prescript{\ast}{}{f\left(\alpha_{1}\right)},\prescript{\ast}{}{g\left(\alpha_{2}\right)}\right)\in\prescript{\ast}{}{\Delta}$,
namely 
\[
f\left(\alpha_{1}\right)<er\left(\alpha_{2}\right)\,\forall f:\mathbb{N}\rightarrow\mathbb{N},
\]
which is one implication in Puritz's characterization. 
\end{example}
\begin{example}
Let $S_{1}=S_{2}=\mathbb{Z}$. Let $\alpha,\beta\in\prescript{\ast}{}{\Z}\setminus\mathbb{Z}$.
Let 
\begin{multline*}
A_{1}=\{z\in\mathbb{Z}\mid z\geq0,z\equiv0\mod2\},\,A_{2}=\{z\in\mathbb{Z}\mid z>0,z\equiv1\mod2\},\\
A_{3}=\{z\in\mathbb{Z}\mid z<0,z\equiv0\mod2\},\,A_{4}=\{z\in\mathbb{Z}\mid z<0,z\equiv1\mod2\}.
\end{multline*}

Let $i,j$ be such that $\alpha\in\prescript{\ast}{}{A_{i}},\beta\in\prescript{\ast}{}{A_{j}}$
and let $f,g\,:\mathbb{Z\rightarrow\mathbb{N}}$ be bijections such
that $f$ coincides with the absolute value on $A_{i}$ and $g$ coincides
with the absolute value on $A_{j}$. Then from conditions (\ref{enu: Main TT Theo (5)})
and (\ref{enu: Main TT Theo (6)}) in Theorem \ref{thm: main TT theo}
we deduce that $\left(\alpha,\beta\right)$ is a tensor pair iff $\left(|\alpha|,|\beta|\right)$
is a tensor pair, namely 
\[
(\alpha,\beta)\,\text{TT}\Leftrightarrow|\alpha|<\prescript{\ast}{}{h}\left(|\beta|\right)\,\forall h:\mathbb{N}\rightarrow\N\,\text{s.t.}\:\prescript{\ast}{}{h}\left(|\beta|\right)\notin\N,
\]
and it is hence straightforward to see that 
\[
\left(\alpha,\beta\right)\,\text{TT}\Leftrightarrow|\alpha|<|\prescript{\ast}{}{h}(\beta)|\,\forall h:\mathbb{Z}\rightarrow\mathbb{Z}\,\text{s.t.}\:\prescript{\ast}{}{h}\left(\beta\right)\notin\mathbb{Z}.
\]
\end{example}
\begin{example}
\label{exa: tensor pairs in Q}Let $S_{1}=S_{2}=\mathbb{Q}$. In $\beta\mathbb{Q}$
there are three kinds of ultrafilters: 
\begin{itemize}
\item principal ones, namely ultrafilters $\U\in\beta\mathbb{Q}$ such that
$\mu(\U)=\{q\}$ for some $q\in\mathbb{Q}$; 
\item quasi-principal, namely ultrafilters $\U\in\beta\mathbb{Q}$ such
that $\mu(\U)$ consists of finite nonstandard elements, in which
case it is very simple to show that there exists $q\in\mathbb{Q}$
such that 
\[
\mu(\U)\subset mon(q)=\left\{ \xi\in\prescript{\ast}{}{\mathbb{Q}}\mid\xi-q\,\text{is infinitesimal}\right\} ;
\]
\item infinite ultrafilters, namely ultrafilters $\U\in\beta\mathbb{Q}$
such that $\mu(\U)$ consists of infinite elements. 
\end{itemize}
Now let $(\alpha,\beta)\in\prescript{\ast}{}{\mathbb{\left(Q\times\mathbb{Q}\right)}}$.
When is it that $\left(\alpha,\beta\right)$ is a tensor pair? As
always, this is the case if $\{\alpha,\beta\}\cap\mathbb{Q}\neq\emptyset.$
If $\{\alpha,\beta\}\cap\mathbb{Q}=\emptyset$, we distinguish three
cases: 
\begin{enumerate}
\item \label{enu: ex1}both $\alpha$ and $\beta$ are infinite; 
\item \label{enu: ex2}both $\alpha$ and $\beta$ are finite; 
\item \label{enu: ex3}one is infinite, one is finite. 
\end{enumerate}
Notice that, as $(\varepsilon,\,^{\ast}\varepsilon$) is a tensor
pair for every infinitesimal $\varepsilon\in\prescript{\ast}{}{\mathbb{Q}}$,
Puritz's characterization does not hold (directly) in our present
case (as $\prescript{\ast}{}{\varepsilon}<\varepsilon$ for every
positive infinitesimal $\varepsilon$).

As, by Theorem \ref{thm: main TT theo}.(\ref{enu: Main TT Theo (5)}),
we know that $(\alpha,\beta)$ is a tensor type iff $(-\alpha,\beta)$
and $(\alpha,-\beta)$ are, we reduce to consider the case $\alpha>0,\beta>0.$
Moreover, let us observe that we can reduce to case (\ref{enu: ex1}).
In fact, if $\eta$ is any finite element in $\prescript{\ast}{}{\mathbb{Q}_{>0}}\setminus\mathbb{Q}$,
let $f_{\eta}:\,\mathbb{Q}\setminus\{st(\eta)\}\rightarrow\mathbb{Q}_{>0}$
be the function such that 
\[
\forall q\in\mathbb{Q}\,f_{\eta}(q)=\frac{1}{q-st(\eta)}.
\]
Then $f_{\eta}(\eta)$ is infinite and, as this function is bijective,
by points (\ref{enu: Main TT Theo (5)}) and (\ref{enu: Main TT Theo (6)})
of Theorem \ref{thm: main TT theo} we get that 
\begin{itemize}
\item if $\alpha$ is finite then $\left(\alpha,\beta\right)$ is TT iff
$\left(f_{\alpha}(\alpha),\beta\right)$ is TT; 
\item if $\beta$ is finite then $\left(\alpha,\beta\right)$ is TT iff
$\left(\alpha,f_{\beta}(\beta)\right)$ is TT. 
\end{itemize}
So we are left to study case (\ref{enu: ex1}). As a simple necessary
criterion, from Theorem \ref{thm: main TT theo}.(\ref{enu: Main TT Theo (5)})
we get that if $\left(\alpha,\beta\right)$ is TT then also the pair
of hypernatural parts $\left(\left\lfloor \alpha\right\rfloor ,\left\lfloor \beta\right\rfloor \right)$
is TT. This fact can be refined: as 
\[
\Delta_{\mathbb{Q}}=\left\{ \left(q_{1},q_{2}\right)\in\mathbb{Q}^{2}\mid q_{2}>q_{1}\right\} \in\mathfrak{U}_{\alpha}\otimes\mathfrak{U}_{\beta}
\]
whenever $\alpha,\beta$ are positive and infinite, we get from Theorem
\ref{thm: main TT theo}.(\ref{enu: Main TT Theo (5)}) that it must
be $\alpha<er_{\mathbb{Q}_{>0}}(\beta)$, where 
\[
er_{\mathbb{Q}_{>0}}(\beta)=\left\{ \prescript{\ast}{}{f}(\beta)\mid f:\mathbb{Q}_{>0}\rightarrow\mathbb{Q}_{>0},\prescript{\ast}{}{f}(\beta)\,\text{is infinite}\right\} .
\]
Let us show that the converse holds as well. Let $\alpha<er_{\mathbb{Q}_{>0}}(\beta)$.
By contrast, assume that $(\alpha,\beta)$ is not TT. Then by Theorem
\ref{thm: main TT theo}.(\ref{enu: Main TT Theo (4)}) there exists
$A\subseteq\mathbb{Q}^{2}$ such that $(q,\beta)\in\prescript{\ast}{}{A}$
for every $q\in\mathbb{Q}$ but $\left(\alpha,\beta\right)\notin\prescript{\ast}{}{A}$.
Let $f:\mathbb{Q}_{>0}\rightarrow\mathbb{Q}_{>0}$ be such that 
\[
\forall q\in\mathbb{Q}\,f(q):=\min\left\{ n\in\mathbb{N}\mid\exists s\in\mathbb{Q}_{>0}\,\left(s<n+1\right)\,\text{and}\,(s,q)\notin A\right\} .
\]
As $(q,\beta)\in\prescript{\ast}{}{A}$ for every $q\in\mathbb{Q}$
we have that $\prescript{\ast}{}{f}(\beta)$ is infinite. And, as
$(\alpha,\beta)\notin\prescript{\ast}{}{A}$, we have that $\prescript{\ast}{}{f}(\beta)\leq\alpha$,
which is absurd. 
\end{example}
\begin{example}
A similar proof can be used to show that, for every infinite $\alpha,\beta\in\prescript{\ast}{}{\mathbb{R}_{>0}}$,
$(\alpha,\beta)$ is TT iff $\alpha<er_{\mathbb{R}_{>0}}(\beta),$
where 
\[
er_{\mathbb{R}_{>0}}(\beta)=\left\{ \prescript{\ast}{}{f}(\beta)\mid f:\mathbb{R}_{>0}\rightarrow\mathbb{R}_{>0},\prescript{\ast}{}{f}(\beta)\,\text{is infinite}\right\} ,
\]
and following ideas similar to those of Example \ref{exa: tensor pairs in Q}
this can be used to characterize tensor pairs in $\mathbb{R}^{2}$.
This can be used also to characterize certain ultrafilters in $\beta\mathbb{C}$:
as $\mathbb{C\cong\mathbb{R}}^{2}$, for example we have that ultrafilters
in $\beta\mathbb{C}$ of the form $\U\oplus i\V$, with $\U,\V\in\beta\mathbb{R}$,
are generated by hypercomplex numbers of the form $\alpha+i\beta$
where $\left(\alpha,\beta\right)$ is TT in $\mathbb{R}^{2}$.

Moreover, as (assuming the continuum hypothesis) $\mathcal{F}:=\N^{\N}$
is in bijection with $\mathbb{R}$, from Theorem \ref{thm: main TT theo}.(\ref{enu: Main TT Theo (6)})
we get a characterization of tensor pairs in $\mathcal{F}^{2}$ and,
since $\N$ can be embedded in $\mathcal{F}$ just mapping any natural
number $n$ to the constant function with value $n$, this gives a
characterization of tensor pairs in $\N\times\mathcal{F}$ and $\mathcal{F}\times\N$.
This characterization is quite implicit; however, Theorem \ref{thm: main TT theo}
can be used to give explicit necessary and sufficient conditions even
in this case: in fact, for $\alpha\in\prescript{\ast}{}{\N}$ and
$\varphi\in\prescript{\ast}{}{\mathcal{F}}$ we have that 
\end{example}
\begin{itemize}
\item Necessary: if $(\alpha,\varphi)$ is TT then $\left(\prescript{\ast}{}{f}(\alpha),\prescript{\ast}{}{H}(\varphi)\right)$
is TT for every $f\in\mathcal{F},\,H:\mathcal{F}\rightarrow\N$. In
particular, by letting for every $n\in\N$ $H_{n}$ be the evaluation
in $\N$, we get that if $(\alpha,\varphi)$ is TT then $(\alpha,\varphi(n))$
is TT in $^{\ast}\left(\N^{2}\right)$ for every $n\in\N$. 
\item Sufficient: $\left(\alpha,\prescript{\ast}{}{\varphi}\right)$ is
TT. In particular, if we let $\V:=\mathfrak{U}_{\alpha}\otimes_{\mathcal{F}}\mathfrak{U}_{\varphi}\in\beta\N$
be the ultrafilter such that $\forall A\subseteq\N$ 
\[
A\in\V\Leftrightarrow\left\{ n\in\N\mid\left\{ f\in\mathcal{F\mid}f(n)\in A\right\} \in\mathfrak{U}_{\varphi}\right\} \in\mathfrak{U}_{\alpha},
\]
we get that $\left(\prescript{\ast}{}{\varphi}\right)(\alpha)\in\mu(\V)$
(these ultrafilters are studied in \cite{21}, where they are used
to study several Ramsey-theoretical combinatorial properties of $\N$). 
\end{itemize}

\subsection{Tensor $k$-ples}

If we consider products of $k$ sets, the natural generalization of
tensor pairs are tensor $k$-ples. 
\begin{defn}
Let $S_{1},\dots,S_{k}$ be sets and, for every $i\leq k$, let $\U_{i}\in\beta S_{i}$.
The tensor product $\U_{1}\otimes\dots\otimes\U_{k}$ is the unique
ultrafilter on $S_{1}\times\dots\times S_{k}$ such that, for every
$A\subseteq S_{1}\times\dots\times S_{k}$ we have that $A\in\U_{1}\otimes\dots\otimes\U_{k}$
if and only if 
\[
\left\{ s_{1}\in S_{1}\mid\left\{ s_{2}\in S_{2}\mid\dots\left\{ s_{k}\in S_{k}\mid\left(s_{1},\dots,s_{k}\right)\in A\right\} \in\U_{k}\right\} \dots\right\} \in\U_{1}.
\]

We say that $\left(\alpha_{1},\dots,\alpha_{k}\right)\in\prescript{\ast}{}{\left(S_{1}\times\dots\times S_{k}\right)}$
is a tensor $k$-ple if $\mathfrak{U}_{\left(\alpha_{1},\dots,\alpha_{k}\right)}=\mathfrak{U}_{\alpha_{1}}\otimes\dots\otimes\mathfrak{U}_{\alpha_{k}}$. 
\end{defn}
It is immediate to prove that $\U_{1}\otimes\dots\otimes\U_{k}$ is
an ultrafilter and that the operation $\otimes$ is associative (modulo
the usual identification of products $\left(S_{1}\times S_{2}\right)\times S_{3}=S_{1}\times\left(S_{2}\times S_{3}\right)=S_{1}\times S_{2}\times S_{3})$,
see e.g. \cite[Appendix]{polyextensions}. This allows to characterize
tensor $k$-ples in terms of pairs: 

\begin{thm}
\label{thm: tensor kples}Let $k\geq1$, let $S_{1},\dots,S_{k},S_{k+1}$
be given sets and let $\left(\alpha_{1},\dots,\alpha_{k+1}\right)\in\prescript{\ast}{}{\left(S_{1}\times\dots\times S_{k+1}\right)}$.
The following facts are equivalent: 
\begin{enumerate}
\item \label{enu: tensor k-ple (1)}$\left(\alpha_{1},\dots,\alpha_{k+1}\right)$
is a tensor $(k+1)$-ple; 
\item \label{enu: tensor k-ple (2)}$\left(\left(\alpha_{1},\dots,\alpha_{k}\right),\alpha_{k+1}\right)$
is a tensor pair and $\left(\alpha_{1},\dots,\alpha_{k}\right)$ is
a tensor $k$-ple; 
\item \label{enu: tensor k-ple (3)}$\left(\alpha_{k},\alpha_{k+1}\right)$
is a tensor pair and $\left(\alpha_{1},\dots,\alpha_{k}\right)$ is
a tensor $k$-ple; 
\item \label{enu: tensor k-ple (4)}$\forall i\leq k$ $\left(\alpha_{i},\alpha_{i+1}\right)$
is a tensor pair. 
\end{enumerate}
\end{thm}
\begin{proof}
(\ref{enu: tensor k-ple (1)})$\Rightarrow$(\ref{enu: tensor k-ple (2)})
By hypothesis, $\mathfrak{U}_{\left(\alpha_{1},\dots,\alpha_{k+1}\right)}=\mathfrak{U}_{\alpha_{1}}\otimes\dots\otimes\mathfrak{U}_{\alpha_{k+1}}$.
By the associativity of tensor products, $\mathfrak{U}_{\alpha_{1}}\otimes\dots\otimes\mathfrak{U}_{\alpha_{k+1}}=\left(\mathfrak{U}_{\alpha_{1}}\otimes\dots\otimes\mathfrak{U}_{\alpha_{k}}\right)\otimes\mathfrak{U}_{\alpha_{k+1}}$.
Let $\V=\mathfrak{U}_{\alpha_{1}}\otimes\dots\otimes\mathfrak{U}_{\alpha_{k}}$.
Then $\left(\left(\alpha_{1},\dots,\alpha_{k}\right),\alpha_{k+1}\right)\in\mu(\V\otimes\U)$,
namely $\left(\left(\alpha_{1},\dots,\alpha_{k}\right),\alpha_{k+1}\right)$
is a tensor pair, and $\left(\alpha_{1},\dots,\alpha_{k}\right)\in\mu\left(\mathfrak{U}_{\alpha_{1}}\otimes\dots\otimes\mathfrak{U}_{\alpha_{k}}\right)$,
namely $\left(\alpha_{1},\dots,\alpha_{k}\right)$ is a tensor $k$-ple.

(\ref{enu: tensor k-ple (2)})$\Rightarrow$(\ref{enu: tensor k-ple (1)})
$\mathfrak{U}_{\left(\alpha_{1},\dots,\alpha_{k},\alpha_{k+1}\right)}=\mathfrak{U}_{\left(\alpha_{1},\dots,\alpha_{k}\right)}\otimes\mathfrak{U}_{\alpha_{k+1}}$
as $\left(\left(\alpha_{1},\dots,\alpha_{k}\right),\alpha_{k+1}\right)$
is a tensor pair, and $\mathfrak{U}_{\left(\alpha_{1},\dots,\alpha_{k}\right)}=\mathfrak{U}_{\alpha_{1}}\otimes\dots\otimes\mathfrak{U}_{\alpha_{k}}$
as $\left(\alpha_{1},\dots,\alpha_{k}\right)$ is a tensor $k$-ple.

(\ref{enu: tensor k-ple (2)})$\Rightarrow$(\ref{enu: tensor k-ple (3)})
By contrast, assume that $\left(\alpha_{k},\alpha_{k+1}\right)$ is
not a tensor pair. Let $A\subseteq S_{k}\times S_{k+1}$ be such that
$\forall s_{k}\in S_{k}\,\left(s_{k},\alpha_{k+1}\right)\in\prescript{\ast}{}{A}$
but $\left(\alpha_{k},\alpha_{k+1}\right)\notin\prescript{\ast}{}{A}$.
Let 
\[
B=\left\{ \left(s_{1},\dots,s_{k+1}\right)\in S_{1}\times\dots\times S_{k+1}\mid\left(s_{k},s_{k+1}\right)\in A\right\} .
\]
By construction, $\forall\left(s_{1},\dots,s_{k}\right)\in S_{1}\times\dots\times S_{k}\,\left(\left(s_{1},\dots,s_{k}\right),\alpha_{k+1}\right)\in\prescript{\ast}{}{B}$.
As $\left(\left(\alpha_{1},\dots,\alpha_{k}\right),\alpha_{k+1}\right)$
is a tensor pair, this entails that $\left(\left(\alpha_{1},\dots,\alpha_{k}\right),\alpha_{k+1}\right)\in\prescript{\ast}{}{B}$,
hence $\left(\alpha_{k},\alpha_{k+1}\right)\in\prescript{\ast}{}{A}$,
which is absurd.

(\ref{enu: tensor k-ple (3)})$\Rightarrow$(\ref{enu: tensor k-ple (2)})
By contrast, assume that $\left(\left(\alpha_{1},\dots,\alpha_{k}\right),\alpha_{k+1}\right)$
is not a tensor pair. Let $T=S_{1}\times\dots\times S_{k}$ and let
$A\subseteq T\times S_{k+1}$ be such that $\forall t\in T\,\left(t,\alpha_{k+1}\right)\in\prescript{\ast}{}{A}$
but $\left(\left(\alpha_{1},\dots,\alpha_{k}\right),\alpha_{k+1}\right)\notin\prescript{\ast}{}{A}$.
Let $B\subseteq S_{k}\times S_{k+1}$ be the set 
\begin{multline*}
B=\{\left(s_{k},s_{k+1}\right)\in S_{k}\times S_{k+1}\mid\forall\left(s_{1},\dots,s_{k-1}\right)\in S_{1}\times\dots\times S_{k-1}\,\\
\left(\left(s_{1},\dots,s_{k}\right),s_{k+1}\right)\in A\}.
\end{multline*}
By construction, $\forall s_{k}\in S_{k}\,\left(s_{k},s_{k+1}\right)\in\prescript{\ast}{}{B}$
hence, as $\left(\alpha_{k},\alpha_{k+1}\right)$ is TT, we have that
$\left(\alpha_{k},\alpha_{k+1}\right)\in\prescript{\ast}{}{B}$, where
\begin{multline*}
\prescript{\ast}{}{B}=\{\left(\eta_{k},\eta_{k+1}\right)\in\prescript{\ast}{}{S_{k}\times S_{k+1}}\mid\forall\left(\sigma_{1},\dots,\sigma_{k-1}\right)\in\prescript{\ast}{}{S_{1}\times\dots\times S_{k-1}}\,\\
\left(\left(\sigma_{1},\dots,\sigma_{k-1},\eta_{k}\right),\eta_{k+1}\right)\in\prescript{\ast}{}{A}\},
\end{multline*}
hence $\left(\left(\alpha_{1},\dots,\alpha_{k}\right),\alpha_{k+1}\right)\in\prescript{\ast}{}{A}$,
which is absurd.

(\ref{enu: tensor k-ple (3)})$\Rightarrow$(\ref{enu: tensor k-ple (4)})
By induction on $k$. If $k=1$ there is nothing to prove. Now let
us assume the claim to hold for $k\geq1$ and let us prove it for
$k+1$. By inductive hypothesis, as (\ref{enu: tensor k-ple (3)})$\Leftrightarrow$(\ref{enu: tensor k-ple (1)}),
$\left(\alpha_{1},\dots,\alpha_{k}\right)$ is a tensor $k$-ple if
and only if $\forall i\leq k-1$ $\left(\alpha_{i},\alpha_{i+1}\right)$
is a tensor pair, so the claim is proven.

(\ref{enu: tensor k-ple (4)})$\Rightarrow$(\ref{enu: tensor k-ple (3)})
This is immediate by induction.
\end{proof}
\begin{example}
If $S_{i}=\N$ for every $i=1,\dots,k$, we get the following equivalence:
if $\forall i\leq k\,\alpha_{i}\in\prescript{\ast}{}{\N}\setminus\N$
then $\left(\alpha_{1},\dots,\alpha_{k+1}\right)$ is a tensor $(k+1)$-ple
if and only if $\alpha<er\left(\alpha_{i+1}\right)$ for every $i\leq k$. 
\end{example}
Notice that, as a trivial corollary of Theorem \ref{thm: tensor kples},
we obtain that the relation of ``being a tensor pair'' is transitive:
\begin{cor}
For every $\left(\alpha_{1},\alpha_{2},\alpha_{3}\right)\in\prescript{\ast}{}{\left(S_{1}\times S_{2}\times S_{3}\right)}$,
if $\left(\alpha_{1},\alpha_{2}\right)$ and $\left(\alpha_{2},\alpha_{3}\right)$
are TT then $\left(\alpha_{1},\alpha_{3}\right)$ is TT.
\end{cor}
\begin{proof}
As $\left(\alpha_{1},\alpha_{2}\right)$ and $\left(\alpha_{2},\alpha_{3}\right)$
are TT, from Theorem \ref{thm: tensor kples} we deduce that $\left(\left(\alpha_{1},\alpha_{2}\right),\alpha_{3}\right)$
is TT. Let us now assume that $\left(\alpha_{1},\alpha_{3}\right)$
is not TT. Let $A\subseteq S_{1}\times S_{3}$ be such that $\left(s_{1},\alpha_{3}\right)\in\prescript{\ast}{}{A}$
for every $s_{1}\in S_{1}$ but $\left(\alpha_{1},\alpha_{3}\right)\notin\prescript{\ast}{}{A}$.
Let $B\subseteq S_{1}\times S_{2}\times S_{3}$ be defined as follows:
$\left(s_{1},s_{2},s_{3}\right)\in B$ if and only if $\left(s_{1},s_{3}\right)\in A$.
Then $\left(s_{1},s_{2},\alpha_{3}\right)\in\prescript{\ast}{}{B}$
for every $\left(s_{1},s_{2}\right)\in S_{1}\times S_{2}$, and so
(as $\left(\left(\alpha_{1},\alpha_{2}\right),\alpha_{3}\right)$
is TT) we have that $\left(\left(\alpha_{1},\alpha_{2}\right),\alpha_{3}\right)\in\prescript{\ast}{}{B}$,
therefore $\left(\alpha_{1},\alpha_{3}\right)\in\prescript{\ast}{}{A}$,
which is absurd.
\end{proof}
Using this fact, it is possible to add the following equivalent characterization
to Theorem \ref{thm: tensor kples}:
\begin{thm}
Let $k\geq1$, let $S_{1},\dots,S_{k},S_{k+1}$ be given sets and
let $\left(\alpha_{1},\dots,\alpha_{k+1}\right)\in\prescript{\ast}{}{\left(S_{1}\times\dots\times S_{k+1}\right)}$.
The following facts are equivalent: 
\begin{enumerate}
\item $\left(\alpha_{1},\dots,\alpha_{k+1}\right)$ is a tensor $(k+1)$-ple; 
\item $\forall F=\left\{ i_{1}<\dots<i_{l}\right\} \subseteq\left\{ 1,\dots,k+1\right\} $
$\left(\alpha_{i_{1}},\dots,\alpha_{i_{l}}\right)$ is a tensor $l$-ple.
\end{enumerate}
\end{thm}
\begin{proof}
The implication $(2)\Rightarrow(1)$ is trivial (just set $F=\left\{ 1,\dots,k+1\right\} $). 

To prove the other implication, by the transitivity of the relation
of being TT we have (using the characterization (\ref{enu: tensor k-ple (4)})
of Theorem \ref{thm: tensor kples}) that for every $i\leq l-1$ $\left(\alpha_{i},\alpha_{i+1}\right)$
is TT. Hence from the equivalence (\ref{enu: tensor k-ple (1)})$\Leftrightarrow$(\ref{enu: tensor k-ple (4)})
in Theorem \ref{thm: tensor kples} we deduce that $\left(\alpha_{i_{1}},\dots,\alpha_{i_{l}}\right)$
is a tensor $l$-ple.
\end{proof}
Finally, by iterating inductively the proof of Theorem \ref{thm:Tensor types in iterated hyperextensions},
we obtain the following result: 
\begin{thm}
Let $S_{1},\dots,S_{k+1}\in\mathbb{V}(X)$ be sets, with $S_{1},\dots,S_{k}$
completely coherent. For every $i\leq k+1$ let $\U_{i}\in\beta S_{i}$
and let $\alpha_{i}\in\mu\left(\U_{i}\right)$. Then $\left(\alpha_{1},\prescript{\ast}{}{\alpha_{2}},\dots,S_{k}\left(\alpha_{k+1}\right)\right)\in\mu\left(\U_{1}\otimes\dots\otimes\U_{k+1}\right)$. 
\end{thm}
As a straightforward corollary we get the following characterization
of tensor $k$-ples in iterated hyperextensions: 
\begin{cor}
Let $S_{1},\dots,S_{k+1}\in\mathbb{V}(X)$ be sets, with $S_{1},\dots,S_{k}$
completely coherent. For every $i\leq k+1$ let $\alpha_{i}\in\prescript{\ast}{}{S_{i}}$.
The following facts are equivalent: 
\begin{enumerate}
\item $\left(\alpha_{1},\dots,\alpha_{k+1}\right)$ is TT; 
\item $\left(\alpha_{1},\dots,\alpha_{k+1}\right)\sim_{S_{1}\times\dots\times S_{k+1}}\left(\alpha_{1},\prescript{\ast}{}{\alpha_{2}},\dots,S_{k}\left(\alpha_{k+1}\right)\right).$ 
\end{enumerate}
\end{cor}

\section{\label{sec:Combinatorial-properties-of}Combinatorial properties
of monads}

\subsection{Partition regularity of existential formulas}

In all this section we let $Y\in\mathbb{V}(X)$ be a set. We will
be concerned with the notion of ``partition regularity\footnote{In this paper we will always say ``partition regularity'' meaning
what is sometimes called ``weak partition regularity''. We will
not consider the notion of ``strong partition regularity'', see
also \cite{13} for a discussion of the two notions.}'': 
\begin{defn}
A family $\mathfrak{F}\subseteq\wp(Y)$ is partition regular if for
every $k\in\mathbb{N}$, for every partition $X=A_{1}\cup\dots\cup A_{k}$
there exists $i\leq k$ such that $A_{i}\in\mathfrak{F}$. 
\end{defn}
The relationship between partition regular families and ultrafilters
is a well known fact in combinatorial number theory\footnote{A family of subsets of $Y$ is partition regular if and only if it
contains an ultrafilter on $Y$.}; in \cite{17}, this characterization was expressed by means of properties
of monads in the case of families of witnesses of the partition regularity
of Diophantine equations, a field rich in very interesting open problems. 
\begin{thm}
\label{thm: monads and polynomials}Let $P\left(x_{1},\dots,x_{n}\right)\in\mathbb{Z}\left[x_{1},\dots,x_{n}\right].$
Then following are equivalent: 
\begin{enumerate}
\item the equation $P\left(x_{1},\dots,x_{n}\right)=0$ is partition regular
on $\N,$ namely the family 
\[
\mathfrak{F}_{P}=\left\{ A\subseteq\N\mid\exists a_{1},\dots,a_{n}\in A\,P\left(a_{1},\dots,a_{n}\right)=0\right\} 
\]
is partition regular; 
\item there exists an ultrafilter $\U\in\beta\N$ and generators $\alpha_{1},\dots,\alpha_{n}\in\mu(\U)$
such that $P\left(\alpha_{1},\dots,\alpha_{n}\right)=0$. 
\end{enumerate}
\end{thm}
This characterization has been subsequently used in a series of paper
\cite{09,10,11,12,18,20} to study the partition regularity of several
classes of nonlinear polynomials. In this section we want to show
how this characterization can be extended to study the partition regularity
of several families of subsets of arbitrary sets\footnote{Some results of this section already appeared, in a much weaker form,
in \cite{17}.}.

Let us start with some preliminaries. In all this section, when we
talk about ``formulas'' we mean first order formulas with bounded
quantifiers\footnote{We adopt a slight abuse of language here: the kind of formulas we
work with are those introduced in Definition \ref{def: PR}, which
contain some unbounded quantifiers. However, the notion we are interested
in is that of a set $A\subseteq Y$ witnessing these formulas, and
when we adopt this notion there are no more unbounded quantifiers
to be handled, as every unbounded quantifier $Q_{i}x_{i}$ ($Q_{i}\in\left\{ \forall,\exists\right\} )$
is substituted with $Q_{i}x_{i}\in A$. For this reason, we believe
that this slight abuse of language should not create too much confusion.} in the language of the superstructure $\mathbb{V}(X)$ (see \cite[Chapter 13]{Goldblatt}),
and when we write a formula as $\phi\left(x_{1},...,x_{n}\right)$
we mean that $x_{1},\dots,x_{n}$ are all and only variables appearing
in $\phi$. We say that a formula $\phi\left(x_{1},...,x_{n}\right)$
is totally open if all its variables are free. 
\begin{defn}
\label{def: PR}Let $\phi\left(x_{1},\dots,x_{n},y_{1},\dots,y_{m}\right)$
be a totally open formula, let $S_{1},\dots,S_{m}\in\mathbb{V}(X)$
be sets and, for $i=1,\dots,m$, let $Q_{i}\in\left\{ \exists,\forall\right\} $.
The existential closure of $\phi\left(x_{1},\dots,x_{n},y_{1},\dots,y_{m}\right)$
with constraints $\left\{ Q_{i}y_{i}\in S_{i}\mid i\leq m\right\} $
is the formula 
\begin{multline*}
E_{\overrightarrow{Qy}\in\overrightarrow{S}}\left(\phi\left(x_{1},\dots,x_{n},y_{1},\dots,y_{m}\right)\right):\\
\exists x_{1}\dots\exists x_{n}\,Q_{1}y_{1}\in S_{1}\dots Q_{m}y_{m}\in S_{m}\,\phi\left(x_{1},\dots,x_{n},y_{1},\dots,y_{m}\right).
\end{multline*}
When $m=0$ we will use the notation $E\left(\phi\left(x_{1},\dots,x_{n}\right)\right)$,
and $E\left(\phi\left(x_{1},\dots,x_{n}\right)\right)$ will be called
the existential closure of $\phi\left(x_{1},\dots,x_{n}\right)$.

Similarly, the universal closure of $\phi\left(x_{1},\dots,x_{n},y_{1},\dots,y_{m}\right)$
with constraints $\left\{ Q_{i}y_{i}\in S_{i}\mid i\leq m\right\} $
is the sentence 
\begin{multline*}
U_{\overrightarrow{Qy}\in\overrightarrow{S}}\left(\phi\left(x_{1},\dots,x_{n},y_{1},\dots,y_{m}\right)\right):\\
\forall x_{1}\dots\forall x_{n}\,Q_{1}y_{1}\in S_{1}\dots Q_{m}y_{m}\in S_{m}\,\phi\left(x_{1},\dots,x_{n},y_{1},\dots,y_{m}\right).
\end{multline*}
When $m=0$ we will use the notation $U\left(\phi\left(x_{1},\dots,x_{n}\right)\right)$,
and $U\left(\phi\left(x_{1},\dots,x_{n}\right)\right)$ will be called
the universal closure of $\phi\left(x_{1},\dots,x_{n}\right)$.

Given a totally open formula $\phi\left(x_{1},\dots,x_{n},y_{1},\dots,y_{m}\right)$,
a set of constraints $\left\{ Q_{i}y_{i}\in S_{i}\mid i\leq m\right\} $
and a set $A\subseteq Y$, $E_{\overrightarrow{Qy}\in\overrightarrow{S}}\left(\phi\left(x_{1},\dots,x_{n},y_{1},\dots,y_{m}\right)\right)$
is said to be modeled by $A$ (notation: $A\models E_{\overrightarrow{Qy}\in\overrightarrow{S}}\left(\phi\left(x_{1},\dots,x_{n},y_{1},\dots,y_{m}\right)\right)$)
if the formula 
\[
\exists x_{1}\in A\dots\exists x_{n}\in A\,Q_{1}y_{1}\in S_{1}\dots Q_{m}y_{m}\in S_{m}\,\phi\left(x_{1},\dots,x_{n},y_{1},\dots,y_{m}\right)
\]
holds true. Similarly, we say that $A$ models $U_{\overrightarrow{Qy}\in\overrightarrow{S}}\left(\phi\left(x_{1},\dots,x_{n},y_{1},\dots,y_{m}\right)\right)$
(notation: $A\models U_{\overrightarrow{Qy}\in\overrightarrow{S}}\left(\phi\left(x_{1},\dots,x_{n},y_{1},\dots,y_{m}\right)\right)$)
if the formula 
\[
\forall x_{1}\in A\dots\forall x_{n}\in A\,Q_{1}y_{1}\in S_{1}\dots Q_{m}y_{m}\in S_{m}\,\phi\left(x_{1},\dots,x_{n}\right)
\]
holds true.

$E_{\overrightarrow{y}\in\overrightarrow{S}}\left(\phi\left(x_{1},\dots,x_{n},y_{1},\dots,y_{m}\right)\right)$
(resp. $U_{\overrightarrow{y}\in\overrightarrow{S}}\left(\phi\left(x_{1},\dots,x_{n},y_{1},\dots,y_{m}\right)\right)$)
is said to be partition regular on $Y$ if for every $k\in\mathbb{N}$,
for every partition $Y=A_{1}\cup\dots\cup A_{k}$ there exists $i\leq k$
such that $A_{i}\models E_{\overrightarrow{y}\in\overrightarrow{S}}\left(\phi\left(x_{1},\dots,x_{n},y_{1},\dots,y_{m}\right)\right)$
(resp. $A_{i}\models U_{\overrightarrow{y}\in\overrightarrow{S}}\left(\phi\left(x_{1},\dots,x_{n},y_{1},\dots,y_{m}\right)\right)$). 
\end{defn}
Our main result in this Section is the following Theorem, which generalizes
Theorem \ref{thm: monads and polynomials} to arbitrary existential
formulas and sets with constraints\footnote{We have included in this Theorem also the known characterization of
partition regular families in terms of ultrafilters, providing a new
rather simple nonstandard proof that uses monads.}: 
\begin{thm}
\label{thm: bridge}Let $\phi\left(x_{1},\dots,x_{n},y_{1},\dots,y_{m}\right)$
be a totally open formula and, for $i=1,\dots,m$, let $S_{i}\in\mathbb{V}(X)$
and $Q_{i}\in\left\{ \exists,\forall\right\} $. Let $Y\in\mathbb{V}(X)$.
The following are equivalent: 
\begin{enumerate}
\item \label{enu: pr 1}$E_{\overrightarrow{Qy}\in\overrightarrow{S}}\left(\phi\left(x_{1},\dots,x_{n},y_{1},\dots,y_{m}\right)\right)$
is partition regular on $Y$; 
\item \label{enu: pr 2}$\exists\alpha_{1}\sim_{Y}\dots\sim_{Y}\alpha_{n}\in\prescript{\ast}{}{Y}$
such that the sentence $Q_{1}y_{1}\in\prescript{\ast}{}{S_{1}},\dots,Q_{m}y_{m}\in\prescript{\ast}{}{S_{m}}$
$\prescript{\ast}{}{\phi}\left(\alpha_{1},\dots,\alpha_{n},y_{1},\dots,y_{m}\right)$
holds true; 
\item \label{enu: pr 3}there exists a ultrafilter $\U\in\beta X$ such
that for every set $A\in\U$ we have that 
\[
A\models E_{\overrightarrow{Qy}\in\overrightarrow{S}}\left(\phi\left(x_{1},\dots,x_{n},y_{1},\dots,y_{m}\right)\right).
\]
\end{enumerate}
\end{thm}
\begin{proof}
(\ref{enu: pr 1})$\Rightarrow$(\ref{enu: pr 2}) First, let us fix
a notation. Let $Par(Y)$ be the set of all possible finite partitions
of $Y$. Given partitions $P_{1}\left(Y\right)=A_{1,1}\cup\dots\cup A_{1,k_{1}},\dots$,
$P_{m}\left(Y\right)=A_{m,1}\cup\dots\cup A_{m,k_{m}}$, we let $P\left(P_{1},\dots,P_{m}\right)$
be the partition generated by $P_{1},\dots,P_{m}$, namely the partition
\[
Y=\bigcup_{\left(i_{1},\dots,i_{m}\right)\in K}\bigcap_{1\leq l\leq m}A_{l,i_{l}},
\]
where $K=\left[1,k_{1}\right]\times\dots\times\left[1,k_{m}\right].$
Now, for every partition $P(Y)=A_{1}\cup\dots\cup A_{m}$ let $I_{P(Y)}$
be the set of all partitions refining $P(Y)$, namely\footnote{As usual, we are identifying partitions and functions with finite
image.} 
\[
I_{P(Y)}=\left\{ f:Y\rightarrow[1,k]\mid k\in\mathbb{N},\forall i\leq k\,\exists!j\leq m\,\text{such that}\,f^{-1}(j)\subseteq A_{i}\right\} .
\]
The family $\left\{ I_{P(Y)}\right\} _{P\in Par(Y)}$ has the FIP,
since $I_{P_{1}(Y)}\cap\dots\cap I_{P_{m}(Y)}\supseteq I_{P(P_{1}\dots,P_{m})}$.
By enlarging, there exists a hyperfinite partition $\prescript{\ast}{}{Y}=A_{1}\cup\dots\cup A_{\lambda}$
that refines all finite partitions of $Y$. As $E_{\overrightarrow{Qy}\in\overrightarrow{S}}\left(\phi\left(x_{1},\dots,x_{n},y_{1},\dots,y_{m}\right)\right)$
is partition regular on $X$, by transfer $\prescript{\ast}{}{E_{\overrightarrow{Qy}\in\overrightarrow{S}}}\left(\phi\left(x_{1},\dots,x_{n},y_{1},\dots,y_{m}\right)\right)$
is partition regular on $\prescript{\ast}{}{Y}$, hence there exists
$i\leq\lambda$, $\alpha_{1},\dots,\alpha_{n}\in A_{i}$, such that
$Q_{1}\beta_{1}\in\prescript{\ast}{}{S_{1}},\dots,Q_{m}\beta_{m}\in\prescript{\ast}{}{S_{m}}$
$\prescript{\ast}{}{\phi}\left(\alpha_{1},\dots,\alpha_{n},\beta_{1},\dots,\beta_{m}\right)$
holds true. To conclude the proof, we show that $\alpha_{1}\sim_{Y}\dots\sim_{Y}\alpha_{n}$.
In fact, as $A_{1}\cup\dots\cup A_{\lambda}$ refines all finite partitions
on $Y$, for every $i\leq\lambda$ it is straightforward to show that
the set 
\[
U_{i}=\left\{ A\subseteq Y\mid A_{i}\subseteq\prescript{\ast}{}{A}\right\} 
\]
is an ultrafilter, and so $A_{i}\subseteq\bigcap_{A\in F_{i}}\prescript{\ast}{}{A}=\mu\left(U_{i}\right)$,
hence $\alpha_{1},\dots,\alpha_{n}\in\mu\left(U_{i}\right)$ are all
$\sim_{Y}$-equivalent.

(\ref{enu: pr 2})$\Rightarrow$(\ref{enu: pr 3}) Let $\U$ be the
ultrafilter generated by $\alpha_{1},\dots,\alpha_{n}$. Let $A\in\U$.
By hypothesis, $\mu(\U)\models\prescript{\ast}{}{E_{\overrightarrow{Qy}\in\overrightarrow{S}}\left(\phi\left(x_{1},\dots,x_{n},y_{1},\dots,y_{m}\right)\right)}$
and, since the formula $E_{\overrightarrow{Qy}\in\overrightarrow{S}}\left(\phi\left(x_{1},\dots,x_{n},y_{1},\dots,y_{m}\right)\right)$
is existential in $\overrightarrow{x}=\left(x_{1},\dots,x_{n}\right)$
and $\mu(\U)\subseteq\prescript{\ast}{}{A}$, this entails that $\prescript{\ast}{}{A}\models\prescript{\ast}{}{E_{\overrightarrow{Qy}\in\overrightarrow{S}}\left(\phi\left(x_{1},\dots,x_{n},y_{1},\dots,y_{m}\right)\right)}$,
so we can conclude by transfer.

(\ref{enu: pr 3})$\Rightarrow$(\ref{enu: pr 1}) This is straightforward
from the definitions, as for every finite partition $Y=A_{1}\cup\dots\cup A_{k}$
there exists $i\leq k$ such that $A_{i}\in\U$. 
\end{proof}
\begin{rem}
The characterization of partition regular Diophantine equations is
a particular case of the previous Theorem, where we let $m=0,\,Y=\N$
and, given a polynomial $P\left(x_{1},\dots,x_{n}\right)$, $\phi\left(x_{1},\dots,x_{n}\right)$
is the formula $P\left(x_{1},\dots,x_{n}\right)=0$. 
\end{rem}
\begin{defn}
If $\phi$ is a partition regular formula, every ultrafilter $\U$
such that $\forall A\in\U\,A\models\phi$ will be called a $\phi$-ultrafilter
(in this case, we will also say that $\U$ witnesses $\phi$). We
will write $\U\models\phi$ to mean that $\U$ is a $\phi$ ultrafilter. 
\end{defn}
In particular, the proof of Theorem \ref{thm: bridge} shows that,
for any ultrafilter $\U\in\beta Y$, $\U\models E_{\overrightarrow{Qy}\in\overrightarrow{S}}\left(\phi\left(x_{1},\dots,x_{n},y_{1},\dots,y_{m}\right)\right)$
if and only $\exists\alpha_{1},\dots,\alpha_{n}\in\mu(\U)$ such that
$Q_{1}\beta_{1}\in\prescript{\ast}{}{S_{1}},\dots,Q_{m}\beta_{m}\in\prescript{\ast}{}{S_{m}}$
$\prescript{\ast}{}{\phi}\left(\alpha_{1},\dots,\alpha_{n},\beta_{1},\dots,\beta_{m}\right)$
holds true. 
\begin{example}
\label{exa: homogeneous}(This example appears, in the weaker form
$m=0$, also in \cite[ Theorem 4.2]{21}.) Let $S$ be a semigroup,
and let $\phi\left(x_{1},\dots,x_{n},y_{1},\dots,y_{m}\right)$ be
an homogeneous totally open formula with constraints $Q_{1}R_{1},\dots,Q_{m}R_{m}$,
in the sense that $\forall s_{1},\dots,s_{n},t\in S$ if $Q_{1}r_{1}\in R_{1}\dots Q_{m}r_{m}\in R_{m}\,\phi\left(s_{1},\dots,s_{n},r_{1},\dots,r_{m}\right)$
holds true then $Q_{1}\widetilde{r_{1}}\in R_{1}\dots Q_{m}\widetilde{r_{m}}\in R_{m}\phi\left(t\cdot s_{1},\dots,t\cdot s_{n},r_{1},\dots,r_{m}\right)$
holds true. Then 
\[
I_{\phi}=\left\{ \U\in\beta S\mid\U\models E_{\overrightarrow{Qy}\in\overrightarrow{R}}\left(\phi\left(x_{1},\dots,x_{n},y_{1},\dots,y_{m}\right)\right)\right\} 
\]
is a closed bilateral ideal in $\beta S$. Closure is trivial; now
let $\U\in I_{\phi}$ and $\V\in\beta S$. Let $\alpha_{1}\sim_{S}\dots\sim_{S}\alpha_{n}\in\mu(\U)$
be such that $Q_{1}y_{1}\in\prescript{\ast}{}{R_{1}}\dots Q_{m}y_{m}\in\prescript{\ast}{}{R_{m}}\,\prescript{\ast}{}{\phi}\left(\alpha_{1},\dots,\alpha_{n},y_{1},\dots,y_{m}\right)$
holds, and let $\beta\in\mu(\V)$. Then: 
\begin{itemize}
\item $\U\odot\V\in I_{\phi}$ as, by Corollary \ref{cor:operations semigroups},
$\alpha_{i}\cdot\prescript{\ast}{}{\beta}\in\mu(\U\odot\V)$ for every
$i\leq n$, and $Q_{1}y_{1}\in\prescript{\ast}{}{R_{1}}\dots Q_{m}y_{m}\in\prescript{\ast}{}{R_{m}}\,\prescript{\ast}{}{\phi}\left(\alpha_{1},\dots,\alpha_{n},y_{1},\dots,y_{m}\right)$
holds as $\prescript{\ast}{}{\phi}\left(\alpha_{1},\dots,\alpha_{n},y_{1},\dots,y_{m}\right)$
holds and $\phi$ is homogeneous; 
\item similarly, $\V\odot\U\in I_{\phi}$ as, by Corollary \ref{cor:operations semigroups},
$\beta\cdot\prescript{\ast}{}{\alpha}\in\mu(\V\odot\U)$ for every
$i\leq n$, and $Q_{1}y_{1}\in\prescript{\ast}{}{R_{1}}\dots Q_{m}y_{m}\in\prescript{\ast}{}{R_{m}}\,\prescript{\ast}{}{\phi}\left(\beta\cdot\prescript{\ast}{}{\alpha_{1}},\dots,\beta\cdot\prescript{\ast}{}{\alpha_{n}},y_{1},\dots,y_{m}\right)$
holds as $\prescript{\ast}{}{\phi}\left(\prescript{\ast}{}{\alpha_{1}},\dots,\prescript{\ast}{}{\alpha_{n}},y_{1},\dots,y_{m}\right)$
holds and $\phi$ is homogeneous. 
\end{itemize}
\end{example}
\begin{example}
In \cite{Khalfalah}, A.~Khalfalah and E.~Szemerèdi proved that,
for every polynomial $P(y)$ such that $2\mid P(y)$ for some $y\in\mathbb{Z}$,
the formula $\exists x_{1},x_{2},\exists y\in\mathbb{Z}\,x_{1}+x_{2}=P(y)$
is partition regular\footnote{However, as a consequence of \cite[Theorem 3.10]{11}, if we drop
the constraint $y\in\mathbb{Z}$, the formula $\exists x_{1},x_{2},y\,\,\,x_{1}+x_{2}=P(y)$
is not partition regular on $\mathbb{Z}$.} on $\mathbb{N}$. By Theorem \ref{thm: bridge}, there exists $\alpha_{1}\sim_{\mathbb{Z}}\alpha_{2}$
and $\beta\in\prescript{\ast}{}{\mathbb{Z}}$ such that $\alpha_{1}+\alpha_{2}=P(\beta)$.
Similarly, in \cite{Franzikinatikis} N.~Frantzikinakis and B.~Host
proved the partition regularity of the formulas $\exists x_{1},x_{2}\exists y_{1}\in\mathbb{Z}\,16x_{1}^{2}+9x_{2}^{2}=y_{1}^{2}$
and $\exists x_{1},x_{2}\exists y_{1}\in\mathbb{Z}\,x_{1}^{2}-x_{1}x_{2}+x_{2}^{2}=y_{1}^{2}$.
Once again, by Theorem \ref{thm: bridge}, there exists $\alpha_{1}\sim_{\mathbb{Z}}\alpha_{2}$
and $\beta\in\prescript{\ast}{}{\mathbb{Z}}$ such that $16\alpha_{1}^{2}+9\alpha_{2}^{2}=\beta_{1}^{2}$
and there exists $\eta_{1}\sim_{\mathbb{Z}}\eta_{2}$ and $\mu_{1}\in\prescript{\ast}{}{\mathbb{Z}}$
such that $\eta_{1}^{2}-\eta_{1}\eta_{2}+\eta_{2}^{2}=\mu_{1}^{2}$.
Notice that both these formulas are homogeneous, hence by Example
\ref{exa: homogeneous} we get that the sets of ultrafilters witnessing
them are closed bilateral ideals in $\left(\beta\N,\odot\right)$
(hence, in particular, any ultrafilter in the minimal bilateral ideal
$\overline{K\left(\beta\N,\odot\right)}$ witnesses both of them). 
\end{example}
\begin{example}
In \cite{Moreira}, J.~Moreira solved a long standing open problem,
proving the partition regularity on $\N$ of the formula\footnote{If we drop the constraint $y\in\mathbb{N}$, the problem of the partition
regularity of the formula $\exists x_{1},x_{2},x_{3},x_{4}\,\left(x_{1}+x_{4}=x_{2}\right)\wedge\left(x_{1}\cdot x_{4}=x_{3}\right)$
is still open.} $\exists x_{1},x_{2},x_{3}\exists y\in\mathbb{N\,}\left(x_{1}+y=x_{2}\right)\wedge\left(x_{1}\cdot y=x_{3}\right)$.
By Theorem \ref{thm: bridge}, this entails the existence of an ultrafilter
$\U\in\beta\N$ and $\alpha_{1},\alpha_{2},\alpha_{3}\in\mu(\U),\beta\in\prescript{\ast}{}{\N}$
such that $\alpha_{1}+\beta=\alpha_{2}$ and $\alpha_{1}\cdot\beta=\alpha_{3}$. 
\end{example}
In most cases, however, one is interested in full partition regularity,
namely in the case of Definition \ref{def: PR} where $m=0$. 
\begin{example}
A very well-know fact in combinatorial number theory is that every
idempotent ultrafilter is a Schur ultrafilter, namely it witnesses
the partition regularity of the formula $\exists x,y,z\,x+y=z$ (see
\cite{schur} for the original combinatorial proof of this result,
and \cite{13} for the ultrafilters version). This fact can be seen
directly also as a consequence of Theorem \ref{thm: bridge}. In fact,
let $\U$ be idempotent and let $\alpha\in\mu(\U)$. Then $\prescript{\ast}{}{\alpha}\in\mu_{2}(\U)$
and $\alpha+\prescript{\ast}{}{\alpha}\in\mu_{3}(\U\oplus\U)=\mu_{3}(\U)$
by Corollary \ref{cor:operations semigroups}, hence letting $\alpha_{1}=\alpha,\alpha_{2}=\prescript{\ast}{}{\alpha}$
and $\alpha_{3}=\alpha+\prescript{\ast}{}{\alpha}$ we get the thesis. 
\end{example}
The characterization of partition regularity by means of ultrafilters
allows to use a iterative process to produce new partition regular
formulas. The following is a generalization of \cite[Lemma 2.1]{11},
where this result was framed and proven restricting to the context
of partition regular equations: 
\begin{thm}
\label{thm: iterativity}Let $\phi\left(x,y_{1},\dots,y_{n}\right)$
be a totally open formula, let $S_{1},\dots,S_{m}\in\mathbb{V}(X)$
be sets and, for $i=1,\dots,m$, let $Q_{i}\in\left\{ \exists,\forall\right\} $.
Assume that 
\[
\exists x\,Q_{1}y_{1}\in S_{1}\dots Q_{n}y_{n}\in S_{n}\,\phi\left(x,y_{1},\dots,y_{n}\right)
\]
is a partition regular formula, and that $\U\in\beta Y$ is one of
its witnesses. Then for every set $A\in\U$ the set 
\[
I_{A}(\phi):=\left\{ a\in A\mid Q_{1}y_{1}\in S_{1}\dots Q_{n}y_{n}\in S_{n}\phi\left(a,y_{1},\dots,y_{n}\right)\,\text{holds true}\right\} \in\U.
\]
Moreover, let $\psi\left(x,z_{1}\dots,z_{m}\right)$ be another totally
open formula, let $R_{1},\dots,R_{m}\in\mathbb{V}(X)$ be sets and,
for $i=1,\dots,m$, let $\widetilde{Q_{i}}\in\left\{ \exists,\forall\right\} $.
Assume that $\U$ witnesses also the partition regularity of $\exists x\widetilde{Q}_{1}z_{1}\in R_{1}\dots\widetilde{Q}_{m}z_{m}\in R_{m}\,\psi\left(x,z_{1},\dots,z_{n}\right)$.
Then $\U$ witnesses the formula 
\begin{multline*}
\exists x\,Q_{1}y_{1}\in S_{1}\dots Q_{n}y_{n}\in S_{n}\widetilde{Q_{1}}z_{1}\in R_{1}\widetilde{Q_{m}}z_{m}\in R_{m}\\
\phi\left(x,y_{1},\dots,y_{n}\right)\wedge\psi\left(x,z_{1},\dots,z_{m}\right),
\end{multline*}
which is then partition regular. 
\end{thm}
\begin{proof}
By contrast, assume that there exists $A\in\U$ such that $I_{A}(\phi)\notin\U$.
Then $A\setminus I_{A}\in\U$, but 
\[
A\setminus I_{A}(\phi)\models\neg\left(\exists x\,Q_{1}y_{1}\in S_{1}\dots Q_{n}y_{n}\in S_{n}\,\phi\left(x,y_{1},\dots,y_{n}\right)\right),
\]
hence $\U$ is not a witness of the partition regularity of $\exists x\,Q_{1}y_{1}\in S_{1}\dots Q_{n}y_{n}\in S_{n}\,\phi\left(x,y_{1},\dots,y_{n}\right)$,
which is absurd.

As for the second claim, let $A\in\U$. Then $I_{A}(\phi)$ and $I_{A}(\psi)$
belong to $\U$, hence $I_{A}(\phi)\cap I_{A}(\psi)\in\U$, and 
\begin{multline*}
I_{A}(\phi)\cap I_{A}(\psi)\models\exists x\,Q_{1}y_{1}\in S_{1}\dots Q_{n}y_{n}\in S_{n}\widetilde{Q_{1}}z_{1}\in R_{1}\widetilde{Q_{m}}z_{m}\in R_{m}\\
\phi\left(x,y_{1},\dots,y_{n}\right)\wedge\psi\left(x,z_{1},\dots,z_{m}\right).
\end{multline*}
Since this formula is existential, this entails that 
\begin{multline*}
A\models\exists x\,Q_{1}y_{1}\in S_{1}\dots Q_{n}y_{n}\in S_{n}\widetilde{Q_{1}}z_{1}\in R_{1}\widetilde{Q_{m}}z_{m}\in R_{m}\\
\phi\left(x,y_{1},\dots,y_{n}\right)\wedge\psi\left(x,z_{1},\dots,z_{m}\right),
\end{multline*}
hence our claim is proven. 
\end{proof}
\begin{example}
Let $X=\N$. For every $n\in\N$ let $\phi_{n}$ be the formula 
\[
\phi_{n}\left(x_{1},\dots,x_{n},y_{1},\dots,y_{n}\right):=\bigwedge_{i\leq n}\left(\sum_{j\leq i}x_{j}=y_{j}\right)
\]
and let $E\left(\phi_{n}\right)$ be the existential closure of $\phi_{n}$.
Hence, for every $A\in\N$ we have that $A\models E\left(\phi_{n}\right)$
if and only if $A$ contains a subset $\left\{ a_{1},\dots,a_{n}\right\} $
of $n$ elements such that all ordered sums $a_{1}+a_{2},a_{1}+a_{2}+a_{3}$
and so on lie in $A$. By Schur's Theorem (see \cite{schur}) we know
that $E\left(\phi_{2}\right)$ is partition regular. Let $\U$ be
a $E\left(\phi_{2}\right)$ ultrafilter (which, from now on, we will
call a Schur ultrafilter). We claim that $\forall n\in\mathbb{N}\,\,\U\models E\left(\phi_{n}\right)$.
We prove this by induction on $n$.

If $n=2$, the claim coincides with our hypothesis.

Now let $n>2$, let us suppose the claim true for $n-1$, and let
us prove it for $n$. By hypothesis and by inductive hypothesis, we
have that $\U$ is a Schur and a $E\left(\phi_{n}\right)$-ultrafilter.
In particular, $\U$ witnesses the formulas\footnote{The apparently strange naming of the variables is chosen to make more
transparent the argument, at least in our hopes.} 
\[
\exists z\left(\exists x_{1}\exists x_{2}\,x_{1}+x_{2}=z\right)
\]
and 
\[
\exists z\left(\exists x_{3}\dots\exists x_{n}\exists y_{2}\dots\exists y_{n}\,\left(z=y_{2}\right)\wedge\bigwedge_{i=3}^{n}\left(z+\sum_{j=3}^{i}x_{j}=y_{i}\right)\right),
\]

hence by Theorem \ref{thm: iterativity} $\U$ witnesses the formula
\begin{multline*}
\exists z\left(\exists x_{1}\exists x_{2}\,x_{1}+x_{2}=z\right)\wedge\\
\left(\exists x_{3}\dots\exists x_{n}\exists y_{2}\dots\exists y_{n}\,\left(z=y_{2}\right)\wedge\bigwedge_{i=3}^{n}\left(z+\sum_{j=3}^{i}x_{j}=y_{j}\right)\right)
\end{multline*}
therefore (by renaming the variables and by letting $y_{1}=x_{1}$)
$\U$ witnesses the partition regularity of the formula 
\[
\exists x_{1}\dots\exists x_{n}\exists y_{1}\dots\exists y_{n}\,\bigwedge_{i\leq n}\left(\sum_{j\leq i}x_{j}=y_{j}\right),
\]
as desired. 
\end{example}

\subsection{Partition regularity of arbitrary formulas}

Even if, in most cases, applications regard existential closures of
totally open formulas, characterizations similar to that of Theorem
\ref{thm: bridge} hold also in other cases. 
\begin{cor}
\label{cor: bridge}Let $\phi\left(x_{1},\dots,x_{n},y_{1},\dots,y_{m}\right)$
be a totally open formula and, for $i=1,\dots,m$, let $S_{i}\in\mathbb{V}(X)$
and $Q_{i}\in\left\{ \exists,\forall\right\} $. Let $Y\in\mathbb{V}(X)$
and $\U\in\beta Y$. Then the following conditions are equivalent: 
\begin{enumerate}
\item there is a set $A$ in $\U$ that satisfies $U_{\overrightarrow{Qy}\in\overrightarrow{S}}\left(\phi\left(x_{1},\dots,x_{n},y_{1},\dots,y_{m}\right)\right)$; 
\item for every $\alpha_{1},...,\alpha_{n}$ in $\mu(\U)$ the sentence
$Q_{1}y_{1}\in\prescript{\ast}{}{S_{1}},\dots,Q_{m}y_{m}\in\prescript{\ast}{}{S_{m}}$
$\prescript{\ast}{}{\phi}\left(\alpha_{1},\dots,\alpha_{n},y_{1},\dots,y_{m}\right)$
holds true. 
\end{enumerate}
\end{cor}
\begin{proof}
This is just Theorem \ref{thm: bridge} applied to the existential
closure of $\neg\phi$.
\end{proof}
A useful consequence of Corollary \ref{cor: bridge} is that, in some
cases, the existence of a generator with some property implies that
this property is shared by all other generators: 
\begin{cor}
\label{cor: inversion}Let $\phi\left(x,y_{1},\dots,y_{n}\right)$
be a totally open formula and, for $i=1,\dots,m$, let $S_{i}\in\mathbb{V}(X)$
and $Q_{i}\in\left\{ \exists,\forall\right\} $. Let $Y\in\mathbb{V}(X)$,
and let $\U$ be an ultrafilter in $\beta Y$ that witnesses $E_{\overrightarrow{Qy}\in\overrightarrow{S}}\left(\phi\left(x,y_{1},\dots,y_{m}\right)\right)$.
Then the formula 
\[
\forall\alpha\in\mu(\U)\,\prescript{\ast}{}{\phi}\left(\alpha,y_{1},\dots,y_{n}\right)
\]
 holds true. 
\end{cor}
\begin{proof}
By Theorem \ref{thm: iterativity} the set $I_{Y}\left(\phi\right)=\left\{ a\in Y\mid\phi\left(a,y_{1},\dots,y_{n}\right)\,\text{holds true}\right\} \in\U$,
namely there is a set $Y$ in $\U$ such that $\forall y\in Y\,\phi\left(y,y_{1},\dots,y_{n}\right)$
holds true. The conclusion hence follows straightforwardly from Corollary
\ref{cor: bridge}. 
\end{proof}
\begin{example}
Let $\U\models\exists x,y_{1},y_{2}\,y_{1}+y_{2}=x.$ In particular,
for every set $A\in\U$ we have that $\U$ witnesses $\exists x\,\exists y_{1},y_{2}\in A\,\left(y_{1}+y_{2}=x\right)$.
Hence from Corollary \ref{cor: inversion} we deduce that $\forall\alpha\in\mu(\U)\,\exists\beta_{1},\beta_{2}\in\prescript{\ast}{}{A}$
such that $\alpha=\beta_{1}+\beta_{2}$. By saturation, this entails
that $\forall\alpha\in\mu(\U)\,\exists\beta_{1},\beta_{2}\in\mu(\U)$
such that $\alpha=\beta_{1}+\beta_{2}$.
\end{example}
\begin{example}
Let $Y=\N$. The formulas 
\[
\phi(d,x,y,z):\,\exists x,y,z,d\,y-x=z-y=d
\]
and 
\[
\psi(d,u,v):\,\exists d,u,v\,u+v=d
\]
are both partition regular and homogeneous. Hence from Example \ref{exa: homogeneous}
we deduce that every ultrafilter $\U\in\overline{K\left(\beta\N,\odot\right)}$
(the minimal closed bilateral ideal in the semigroup $\left(\beta\N,\odot\right)$)
witnesses both $\phi(d,x,y,z)$ and $\psi(d,u,v)$. Therefore, by
Corollary \ref{cor: inversion} we get that for every set $A\in\U$
there exists an arithmetic progression in $A$ of length 3 with a
common difference in $A$ that can be written as a sum of elements
of $A$ and, analogously, that every set $A\in\U$ contains elements
$x,y,z$ that are increments in arithmetic progressions of length
3 and such that $x+y=z$. Moreover, if $\U\odot\U=\U\in\overline{K\left(\beta\N,\odot\right)}$
then $\U$ witnesses also the formula $\varphi(d,u,v):\,\exists d,u,v\,u\cdot v=d$,
hence, again by Corollary \ref{cor: inversion}, we get that every
set $A\in\U$ contains an arithmetic progression in $A$ of length
3 with a common difference in $A$ that can be written as a product
of elements of $A$. 
\end{example}
\begin{example}
\label{exa:Selective}Selective ultrafilters admit several equivalent
characterizations (see e.g.~\cite{Blass}). One of them says that
$\U$ is a selective ultrafilter on $Y$ if and only if every function
$f:Y\rightarrow Y$ is $\U$-equivalent to either an injective or
a constant function, namely there exists $A\in Y$ such that $f|_{A}$
is injective or constant. By Corollary \ref{cor: bridge}, this is
equivalent to say that for every $f:Y\rightarrow Y$ the function
$\prescript{\ast}{}{f}$ is injective or constant on $\mu(\U)$.

Let us consider the case $Y=\N$. In this case, it is simple to see
that ``injective'' can be substituted with ``strictly increasing''.
Let $P(x)\in\mathbb{Z}[x]$. Let $\left\{ a_{n}\mid n\in\N\right\} $
be the sequence inductively defined as follows: $a_{0}=0$ and, for
every $n\geq0$, 
\[
a_{n+1}=\min\left\{ m\in\N\mid m>\left|P\left(j\right)\right|\,\forall j\leq a_{n}\right\} .
\]

Let $f_{P}:\N\rightarrow\N$ be the function such that 
\[
\forall m\in\N\,f(m)=\max\left\{ a_{n}\mid a_{n}\leq m\right\} .
\]
As $f^{-1}\left(m\right)$ is finite for every $m\in\N$, there exists
$A\in\U$ such that $f_{P}|_{A}$ is increasing. Hence we have that
\begin{equation}
\forall P(x)\in\mathbb{Z}\left[x\right]\forall\alpha,\beta\in\mu\left(\U\right)\left(\alpha<\beta\right)\Rightarrow\left(\left|P(\alpha)\right|<\beta\right).\label{eq:selective}
\end{equation}
if $\alpha<\beta$ then $\left|P(\alpha)\right|<\beta$. As a consequence,
we have that no selective ultrafilter is Schur: in fact, if $\U$
is a selective Schur ultrafilter, by Theorem \ref{thm: bridge} there
are $\alpha,\beta,\gamma\in\mu(\U)$ such that $\alpha+\beta=\gamma$
and, if $\alpha\geq\beta$, this means that $\alpha<\gamma\leq2\alpha$,
which is in contrast with the characterization (\ref{eq:selective}). 
\end{example}
\begin{example}
The result of Example \ref{exa:Selective} can be generalized. First
of all, from characterization (\ref{eq:selective}) we deduce immediately
the following strengthening: 
\begin{multline}
\forall n\in\N\,\forall P\left(x_{1},\dots,x_{n}\right)\in\mathbb{Z}\left[x_{1},\dots,x_{n}\right]\,\forall\alpha_{1},\dots,\alpha_{n},\beta\in\mu\left(\U\right)\\
\left(\alpha_{1},\dots,\alpha_{n}<\beta\right)\Rightarrow\left(\left|P\left(\alpha_{1},\dots,\alpha_{n}\right)\right|<\beta\right);\label{eq:selettivo2}
\end{multline}
in fact, if $P\left(x_{1},\dots,x_{n}\right)=\sum_{i=1}^{k}c_{i}x_{1}^{e_{1,i}}\cdot\dots\cdot x_{n}^{e_{n,i}}$
then 
\[
\left|P\left(x_{1},\dots,x_{n}\right)\right|=\left|\sum_{i=1}^{k}c_{i}x_{1}^{e_{1,i}}\cdot\dots\cdot x_{n}^{e_{n,j}}\right|\leq\sum_{i=1}^{k}\left|c_{i}x_{1}^{e_{1,i}}\cdot\dots\cdot x_{n}^{e_{n,i}}\right|,
\]
hence if $\alpha=\max\left\{ \alpha_{i}\mid i\leq n\right\} $ then
$\left|P\left(\alpha_{1},\dots,\alpha_{n}\right)\right|\leq\sum_{i=1}^{k}\left|c_{i}\right|\alpha^{\sum_{j=1}^{n}e_{n,j}},$
so we conclude by characterization (\ref{eq:selective}).

Now we use fact (\ref{eq:selettivo2}) to prove that for every polynomial
$P\left(x_{1},\dots,x_{n}\right)\in\mathbb{Z}\left[x_{1},\dots,x_{n}\right]$
and for every selective ultrafilter $\U$, $\U$ is not a witness
of the partition regularity of the formula 
\begin{equation}
\exists x_{1},\dots,x_{n}\,\left(\bigwedge_{1\leq i<j\leq n}x_{i}\neq x_{j}\right)\wedge P\left(x_{1},\dots,x_{n}\right)=0.\label{eq:selettivi3}
\end{equation}
We proceed by induction. If $n=2$, the claim is trivial, as in this
case by Rado's Theorem\footnote{\begin{thm*}[Rado]
A linear polynomial $\sum_{i=1}^{n}c_{i}x_{i}$ is partition regular
if and only if there exists a non empty subset $I\subseteq\left\{ 1,\dots,n\right\} $
such that $\sum_{i\in I}c_{i}=0$.
\end{thm*}
} (see \cite{Rado}) the only partition regular polynomial in two variables
is $x-y$.

Now let $n>2$ and let us assume the claim to be true for $n-1$.
Assume, by contrast, that there exists a polynomial $P\left(x_{1},\dots,x_{n}\right)\in\mathbb{Z}\left[x_{1},\dots,x_{n}\right]$
and a selective ultrafilter $\U$ that witnesses the partition regularity
of the formula (\ref{eq:selettivi3}). Then, by Theorem \ref{thm: bridge}
we can find mutually distinct elements $\alpha_{1},\dots,\alpha_{n}\in\mu(\U)$
such that $P\left(\alpha_{1},\dots,\alpha_{n}\right)=0$. By rearranging
the indexes, if necessary, we can assume that $\alpha_{n}=\max\left\{ \alpha_{i}\mid i\leq n\right\} $.

Let $P\left(x_{1},\dots,x_{n}\right)=\sum_{i=1}^{k}c_{i}x_{1}^{e_{1,i}}\cdot\dots\cdot x_{n}^{e_{n,i}}$,
let $J=\left\{ i\in[1,k]\mid e_{n,i}>0\right\} $, let $Q\left(x_{1},\dots,x_{n}\right)=\sum_{i\in J}c_{i}x_{1}^{e_{1,i}}\cdot\dots\cdot x_{n}^{e_{n,i}}$
and $R\left(x_{1},\dots,x_{n-1}\right)=\sum_{i\notin J}c_{i}x_{1}^{e_{1,i}}\cdot\dots\cdot x_{n-1}^{e_{n-1,i}}$.
As $P\left(\alpha_{1},\dots,\alpha_{n}\right)=0,$ we have that 
\[
\left|Q\left(\alpha_{1},\dots,\alpha_{n}\right)\right|=\left|R\left(\alpha_{1},\dots,\alpha_{n-1}\right)\right|.
\]
From characterization (\ref{eq:selettivo2}) we have that $\left|R\left(\alpha_{1},\dots,\alpha_{n-1}\right)\right|<\alpha_{n}$.
We consider two cases: 
\begin{itemize}
\item If $Q\left(\alpha_{1},\dots,\alpha_{n}\right)\neq0$ then $\left|Q\left(\alpha_{1},\dots,\alpha_{n}\right)\right|\geq\alpha_{n}$,
as $Q\left(x_{1},\dots,x_{n}\right)\in\mathbb{Z}\left[x_{1},\dots,x_{n}\right]$,
$x_{n}\mid Q\left(x_{1},\dots,x_{n}\right)$ and $Q\left(\alpha_{1},\dots,\alpha_{n}\right)\neq0$,
hence it cannot be $\left|Q\left(\alpha_{1},\dots,\alpha_{n}\right)\right|=\left|R\left(\alpha_{1},\dots,\alpha_{n-1}\right)\right|$
and we have reached an absurd; 
\item If $Q\left(\alpha_{1},\dots,\alpha_{n}\right)=0$ then $R\left(x_{1},\dots,x_{n-1}\right)=0$,
and we can conclude by using the inductive hypothesis. 
\end{itemize}
\end{example}

\subsection{Combinatorial properties with internal parameters}

As shown in our examples, Theorem \ref{thm: bridge} can be used to
prove several properties of monads. This result can be strengthened,
in saturated extensions, taking into account also internal parameters: 
\begin{thm}
\label{thm: nonstandard bridge}Let $\overrightarrow{p}:=\left(p_{1},\dots,p_{k}\right)$,
where $p_{1},\dots,p_{k}$ are internal objects in $\mathbb{V}(X)$.
Let $S_{1},\dots,S_{m}$ be internal sets in $\mathbb{V}(X)$ and,
for every $i=1,\dots,m$ let $Q_{i}\in\left\{ \exists,\forall\right\} $.
Let $\U\in\beta Y$ and let $\phi\left(x_{1},\dots,x_{n},y_{1},\dots,y_{m},z_{1},\dots,z_{k}\right)$
be a totally open formula. The following facts are equivalent: 
\begin{enumerate}
\item \label{enu: nonstandard bridge 1}$\forall A\in\U\,\exists\alpha_{1}\dots\alpha_{n}\in\prescript{\ast}{}{A}\,Q_{1}\beta_{1}\in S_{1}\dots Q_{m}\beta_{m}\in S_{m}\,\prescript{\ast}{}{\phi}\left(\overrightarrow{\alpha},\overrightarrow{\beta},\overrightarrow{p}\right)$
holds true, where $\overrightarrow{\alpha}=\left(\alpha_{1},\dots,\alpha_{n}\right)$
and $\overrightarrow{\beta}=\left(\beta_{1},\dots,\beta_{m}\right)$; 
\item \label{enu: nonstandard bridge 2}$\exists\alpha_{1}\dots\alpha_{n}\in\mu(\U)\,Q_{1}\beta_{1}\in S_{1}\dots Q_{1}\beta_{m}\in S_{m}\,\prescript{\ast}{}{\phi}\left(\overrightarrow{\alpha},\overrightarrow{\beta},\overrightarrow{p}\right)$
holds true, where $\overrightarrow{\alpha}=\left(\alpha_{1},\dots,\alpha_{n}\right)$
and $\overrightarrow{\beta}=\left(\beta_{1},\dots,\beta_{m}\right)$. 
\end{enumerate}
\end{thm}
\begin{proof}
(\ref{enu: nonstandard bridge 1})$\Rightarrow$(\ref{enu: nonstandard bridge 2})
For every $A\in\U$ let 
\begin{multline*}
I_{A}=\{\left(\alpha_{1},\dots,\alpha_{m}\right)\in\prescript{\ast}{}{A}^{m}\mid Q_{1}\beta_{1}\in S_{1}\dots Q_{m}\beta_{m}\in S_{m}\\
\prescript{\ast}{}{\phi}\left(\alpha_{1},\dots,\alpha_{n},\beta_{1},\dots,\beta_{m},p_{1},\dots,p_{k}\right)\,\text{holds true}\}.
\end{multline*}
The family $\left\{ I_{A}\right\} _{A\in\U}$ has the finite intersection
property as $I_{A_{1}}\cap I_{A_{2}}=I_{A_{1}\cap A_{2}}$, and every
set $I_{A}$ is internal by the internal definition principle. Hence,
by saturation the formula 
\[
\exists\alpha_{1}\dots\alpha_{n}\in\mu(\U)\,Q_{1}\beta_{1}\in S_{1}\dots Q_{m}\beta_{m}\in S_{m}\,\prescript{\ast}{}{\phi}\left(\alpha_{1},\dots,\alpha_{n},\beta_{1},\dots,\beta_{m},p_{1},\dots,p_{k}\right)
\]
holds true.

(\ref{enu: nonstandard bridge 1})$\Rightarrow$(\ref{enu: nonstandard bridge 2})
Just notice that $\mu(\U)\subseteq\prescript{\ast}{}{A}$ for every
$A\in\U$ by definition. 
\end{proof}
\begin{cor}
\label{cor: nonstandard bridge}Let $\overrightarrow{p}:=\left(p_{1},\dots,p_{k}\right)$
where $p_{1},\dots,p_{k}$ are internal objects in $\mathbb{V}(X)$.
Let $S_{1},\dots,S_{m}$ be internal sets in $\mathbb{V}(X)$ and,
for every $i=1,\dots,m$ let $Q_{i}\in\left\{ \exists,\forall\right\} $.
Let $\U\in\beta Y$ and let $\phi\left(x_{1},\dots,x_{n},y_{1},\dots,y_{m},z_{1},\dots,z_{k}\right)$
be a totally open formula. The following facts are equivalent: 
\begin{enumerate}
\item \label{enu: nonstandard bridge 1-1}$\exists A\in\U\,\forall\alpha_{1}\dots\alpha_{n}\in\prescript{\ast}{}{A}\,Q_{1}\beta_{1}\in S_{1}\dots Q_{m}\beta_{m}\in S_{m}\,\prescript{\ast}{}{\phi}\left(\overrightarrow{\alpha},\overrightarrow{\beta},\overrightarrow{p}\right)$
holds true, where $\overrightarrow{\alpha}=\left(\alpha_{1},\dots,\alpha_{n}\right)$
and $\overrightarrow{\beta}=\left(\beta_{1},\dots,\beta_{m}\right)$; 
\item \label{enu: nonstandard bridge 2-1}$\forall\alpha_{1}\dots\alpha_{n}\in\mu(\U)\,Q_{1}\beta_{1}\in S_{1}\dots Q_{m}\beta_{m}\in S_{m}\,\prescript{\ast}{}{\phi}\left(\overrightarrow{\alpha},\overrightarrow{\beta},\overrightarrow{p}\right)$
holds true, where $\overrightarrow{\alpha}=\left(\alpha_{1},\dots,\alpha_{n}\right)$
and $\overrightarrow{\beta}=\left(\beta_{1},\dots,\beta_{m}\right)$. 
\end{enumerate}
\end{cor}
\begin{proof}
Just apply Theorem \ref{thm: nonstandard bridge} to $\neg\phi\left(x_{1},\dots,x_{n},y_{1},\dots,y_{m},z_{1},\dots,z_{k}\right)$. 
\end{proof}
\begin{example}
Let $X=\mathbb{Q}$. Let $\U$ be a positive infinite ultrafilter
in $\mathbb{\beta Q}$ (in the sense of Example \ref{exa: tensor pairs in Q}).
We claim that $\mu(\U)$ is right and left unbounded in the set $Inf\left(\prescript{\ast}{}{\Q}\right)$
of positive infinite elements of $\prescript{\ast}{}{\Q}$. By contrast,
assume that there are $\beta_{1},\beta_{2}\in Inf\left(\prescript{\ast}{}{\Q}\right)$
such that $\beta_{1}<\alpha<\beta_{2}$ for every $\alpha\in\mu(\U)$.
Then by Corollary \ref{cor: nonstandard bridge} we have that there
exists $A\in\U$ such that $\beta_{1}<\alpha<\beta_{2}$ for every
$\alpha\in\prescript{\ast}{}{A}$. However: 
\begin{itemize}
\item $\beta_{1}$ cannot exist, as $A\subseteq\prescript{\ast}{}{A}$ and
$q<\alpha_{1}$ for every $q\in Inf\left(\prescript{\ast}{}{\Q}\right)$; 
\item $\alpha_{2}$ cannot exist, as every set $B\in\U$ is right unbounded
(and so is $\prescript{\ast}{}{B}$ by transfer). 
\end{itemize}
\end{example}
\begin{example}
Let $X=\N^{\N}$. Let $\U$ be an ultrafilter in $\beta X$ and let
$\alpha_{1},\alpha_{2}\in\prescript{\ast}{}{\N}$. Then every generator
$\varphi$ of $\U$ maps $\alpha_{1}$ into $\alpha_{2}$ if and only
if there is a set $B\in\U$ such that every function in $B$ maps
$\alpha_{1}$ into $\alpha_{2}$. For example, if $\alpha_{1}\in\N$
and $\alpha_{2}\in\,^{\ast}\N\setminus\N$ this means that no ultrafilter
has this property, as if a function $f\in B$ then $\prescript{\ast}{}{f}\left(\alpha_{1}\right)\in\N$. 
\end{example}
We conclude by considering another version of the partition regularity
of properties where multiple ultrafilters are considered at once\footnote{Similar ideas, but in a rather different context, appeared in \cite{Blass2}.}. 
\begin{example}
\label{exa: finite embed}In \cite{04,17,19} it has been introduced
and studied the notion of finite embeddability between subsets of
$\N$. In \cite{21}, this notion has been extended to arbitrary families
of functions and semigroups. In particular, if $\left(S,\cdot\right)$
is a commutative\footnote{In a very similar way, we can work with non-commutative semigroups;
however, this means considering the different notions of right and
left finite embeddability, and we prefer to avoid such complications
here.} semigroup, a set $A\subseteq S$ is finitely embeddable in a set
$B\subseteq S$ (notation: $A\leq_{fe}B$) iff for every finite subset
$F\subseteq S$ there exists $t\in S$ such that $t\cdot F\subseteq B$.
If we fix the cardinality $n$ of the finite set $F$, we can rewrite
this property as 
\[
\forall a_{1},\dots,a_{n}\in A\,\exists b_{1},\dots,b_{n}\in B\,\exists t\in S\,\bigwedge_{i\leq n}\left(a_{i}\cdot t=b_{i}\right).
\]
This notion has been extended to ultrafilters in \cite{19}: a ultrafilter
$\U\in\beta S$ is finitely embeddable in $\V\in\beta S$ (notation:
$\U\leq_{fe}\V$) if and only if for every set $B\in\V$ there exists
$A\in\U$ such that $A\leq_{fe}B$. Once again, if in $\leq_{fe}$
we fix the cardinality of the finite sets to be embedded, we can rewrite
the finite embeddability between ultrafilters as follows: 
\[
\forall A\in\V\,\exists B\in\,\U\,\forall a_{1},\dots,a_{n}\in A\,\exists b_{1},\dots,b_{n}\in B\,\exists t\in S\,\bigwedge_{i\leq n}\left(a_{i}+t=b_{i}\right).
\]
\end{example}
We want to give a nonstandard characterization of properties like
that expressed in Example \ref{exa: finite embed}. For the sake of
simplicity, we give it for an alternation $\forall-\exists$ of two
ultrafilters; similar characterizations for arbitrary finite amounts
of ultrafilters and different alternations of quantifiers can be analogously
deduced. 
\begin{thm}
\label{thm: nonstandard bridge-1}Let $\overrightarrow{p}:=\left(p_{1},\dots,p_{k}\right)$,
where $p_{1},\dots,p_{k}$ are internal objects in $\mathbb{V}(X)$.
Let $S_{1},\dots,S_{h}$ be internal sets in $\mathbb{V}(X)$. Let
$\U,\V\in\beta Y$ and let $\phi\left(x_{1},\dots,x_{n},y_{1},\dots,y_{m},t_{1},\dots,t_{h},z_{1},\dots,z_{k}\right)$
be a totally open formula. Assume that the extension $\prescript{\ast}{}{Y}$
is $|Y|^{+}$-saturated. The following facts are equivalent: 
\begin{enumerate}
\item \label{enu: nonstandard bridge 1-2}$\forall A\in\U\exists B\in\V\,\forall\beta_{1},\dots,\beta_{m}\in\prescript{\ast}{}{B}\,\exists\alpha_{1},\dots,\alpha_{n}\in\prescript{\ast}{}{A}\,\exists s_{1}\in S_{1}\dots\exists s_{h}\in S_{h}\,\prescript{\ast}{}{\phi}\left(\overrightarrow{\alpha},\overrightarrow{\beta},\overrightarrow{s},\overrightarrow{p}\right)$,
where $\overrightarrow{\alpha}=\left(\alpha_{1},\dots,\alpha_{n}\right),\overrightarrow{\beta}=\left(\beta_{1},\dots,\beta_{m}\right),\overrightarrow{s}=\left(s_{1},\dots,s_{h}\right),$
$\overrightarrow{p}=\left(p_{1},\dots,p_{k}\right)$; 
\item \label{enu: nonstandard bridge 2-2}$\forall\beta_{1}\dots\beta_{n}\in\mu(\U)\,\exists\alpha_{1}\dots\alpha_{m}\in\mu\left(\V\right)\,\exists s_{1}\in S_{1}\dots\exists s_{h}\in S_{h}\,\prescript{\ast}{}{\phi}\left(\overrightarrow{\alpha},\overrightarrow{\beta},\overrightarrow{s},\overrightarrow{p}\right)$
holds true, where $\overrightarrow{\alpha}=\left(\alpha_{1},\dots,\alpha_{n}\right),\overrightarrow{\beta}=\left(\beta_{1},\dots,\beta_{m}\right),\overrightarrow{s}=\left(s_{1},\dots,s_{h}\right),\overrightarrow{p}=\left(p_{1},\dots,p_{k}\right)$. 
\end{enumerate}
\end{thm}
\begin{proof}
We will use the notations $\overrightarrow{\alpha}=\left(\alpha_{1},\dots,\alpha_{n}\right),\overrightarrow{\beta}=\left(\beta_{1},\dots,\beta_{m}\right),\overrightarrow{s}=\left(s_{1},\dots,s_{h}\right),\overrightarrow{p}=\left(p_{1},\dots,p_{k}\right)$
throughout the proof.

(\ref{enu: nonstandard bridge 1-2})$\Rightarrow$(\ref{enu: nonstandard bridge 2-2})
Let $\overrightarrow{\beta}\in\mu(\U)^{n}$. As $\mu(\V)\subseteq\prescript{\ast}{}{B}$
for every $B\in\V$, we have that for every $A\in\U$ the set 
\[
I_{A}:=\left\{ \overrightarrow{\alpha}\in\prescript{\ast}{}{A}^{n}\mid\exists s_{1}\in S_{1}\dots\exists s_{h}\in S_{h}\,\prescript{\ast}{}{\phi}\left(\overrightarrow{\alpha},\overrightarrow{\beta},\overrightarrow{s},\overrightarrow{p}\right)\,\text{holds true}\right\} \neq\emptyset.
\]

As $I_{A}$ is internal and $\left\{ I_{A}\right\} _{A\in\U}$ has
the FIP, by saturation $|Y|^{+}$-saturation we have that $\bigcap_{A\in\U}I_{A}\neq\emptyset$,
and we conclude as if $\bigcap_{A\in\U}I_{A}\subseteq\mu(\U)^{n}$.

$(2)\Rightarrow(1)$ Let $A\in\U$. By using $\prescript{\ast}{}{A}$
as a parameter, we see that the thesis is a straightforward consequence
of Corollary \ref{cor: nonstandard bridge}. 
\end{proof}
\begin{example}
Let us consider the finite embeddability. Let $\left(S,\cdot\right)$
be a commutative semigroup with identity and let $\U,\V\in\beta S$.
From Theorem \ref{thm: nonstandard bridge-1} we deduce that, for
every $n\in\N$, the following two conditions are equivalent: 
\begin{itemize}
\item $\U\leq_{fe}\V$; 
\item $\forall\beta_{1}\dots\beta_{n}\in\mu(\U)\exists\sigma\in\prescript{\ast}{}{S}$
such that $\sigma\cdot\beta_{1},\dots,\sigma\cdot\beta_{n}\in\mu(\V)$. 
\end{itemize}
In particular, as $\V\in\beta S$ is such that $\forall\U\in\beta S\,\U\leq_{fe}\V$
if and only if $\V\in\overline{K(\beta S,\odot)}$ (this result has
been proven in \cite[Theorem 4.13]{21}), we obtain the equivalence
between the following two properties: 
\begin{itemize}
\item $\V\in\overline{K(\beta S,\odot)};$ 
\item $\forall n\in\N$, $\forall\beta_{1}\sim_{S}\dots\sim_{S}\beta_{n}\in\prescript{\ast}{}{S}\,\exists\sigma\in\prescript{\ast}{}{S}$
such that $\sigma\cdot\beta_{1},\dots,\sigma\cdot\beta_{n}\in\mu(\V)$. 
\end{itemize}
\end{example}
Finally, as $\V\in\overline{K(\beta S,\odot)}$ if and only if every
set $A\in\V$ is piecewise syndetic in $\left(S,\cdot\right)$ (see
e.g. \cite[Theorem 4.40]{13}), from Theorem \ref{thm: nonstandard bridge}
we obtain the following characterization\footnote{Notice that this characterization resembles that of thick subsets
of $S$: a set $A\subseteq S$ is thick if and only if for every $s_{1},\dots,s_{n}\in S$
there exists $t\in S$ such that $t\cdot s_{1},\dots t\cdot s_{n}\in A$,
i.e. (by transfer) if for every $\beta_{1},\dots,\beta_{n}\in\prescript{\ast}{}{S}$
there exists $\sigma\in\prescript{\ast}{}{S}$ such that $\sigma\cdot\beta_{1},\dots\sigma\cdot\beta_{n}\in\prescript{\ast}{}{A}$.} of piecewise syndetic subsets of $S$: $A\subseteq S$ is piecewise
syndetic if and only if 
\[
\forall n\in\N,\forall\beta_{1}\sim_{S}\dots\sim_{S}\beta_{n}\in\prescript{\ast}{}{S}\,\exists\sigma\in\prescript{\ast}{}{S}\,\text{such that}\,\sigma\cdot\beta_{1},\dots,\sigma\cdot\beta_{n}\in\prescript{\ast}{}{S}.
\]

\begin{example}
Finite embeddabilities can be generalized to arbitrary families of
functions $\mathcal{F}:S^{n}\rightarrow S$ (see \cite{21}). In particular,
let $S=\N$ and $\mathcal{F}:\N\rightarrow\N$ be the family of affinities
\[
\mathcal{F}:=\left\{ f_{a,b}:\N\rightarrow\N\mid\forall n\in\N\,f_{a,b}(n)=an+b\right\} .
\]

We say that a set $A\subseteq\N$ is $\mathcal{F}$-finitely embeddable
in $B\subseteq\N$ (notation: $A\leq_{\mathcal{F}}B$) if for every
finite set $F\subseteq A$ there exists $f\in\mathcal{F}$ such that
$f(A)\subseteq B$. Of course, this notion is related to that of AP-rich
set (namely, of a set that contains arbitrarily long arithmetic progressions):
in fact, it is straightforward to see that $B\subseteq\N$ is AP-rich
if and only if $A\leq_{\mathcal{F}}B$ for every $A\subseteq\N$.
$\mathcal{F}$-finite embeddability can be extended to ultrafilters
as follows: we say that an ultrafilter $\U\in\beta\N$ is $\mathcal{F}$-finitely
embeddable in $\V\in\beta\N$ if for every set $B\in\V$ there exists
$A\in\U$ such that $A\leq_{\mathcal{F}}B$. Again, from Theorem \ref{thm: nonstandard bridge-1}
we deduce that, for every $n\in\N$, the following two conditions
are equivalent: 
\begin{itemize}
\item $\U\leq_{\mathcal{F}}\V$; 
\item $\forall\beta_{1}\dots\beta_{n}\in\mu(\U)\exists\sigma,\rho\in\prescript{\ast}{}{\N}$
such that $\sigma\cdot\beta_{1}+\rho,\dots,\sigma\cdot\beta_{n}+\rho\in\mu(\V)$. 
\end{itemize}
In \cite{21} we proved that $\V\in\beta\N$ is such that $\forall\U\in\beta\N\,\U\leq_{fe}\V$
if and only if every set $A\in\V$ is AP-rich. In particular, Theorem
\ref{thm: nonstandard bridge-1} entails the equivalence between the
following two properties: 
\begin{itemize}
\item every set $A\in\V$ is AP-rich; 
\item $\forall n\in\N$, $\forall\beta_{1}\sim_{S}\dots\sim_{S}\beta_{n}\in\prescript{\ast}{}{\N}\,\exists\sigma,\rho\in\prescript{\ast}{}{\N}$
$\sigma\cdot\beta_{1}+\rho,\dots,\sigma\cdot\beta_{n}+\rho\in\mu(\V)$. 
\end{itemize}
\end{example}
In particular, as the family of AP-rich sets is strongly partition
regular, from Theorem \ref{thm: nonstandard bridge} we obtain the
following characterization of AP-rich sets: $A\subseteq\N$ is AP-rich
if and only if 
\[
\forall n\in\N,\forall\beta_{1}\sim_{S}\dots\sim_{S}\beta_{n}\in\prescript{\ast}{}{\N}\,\exists\sigma,\rho\in\prescript{\ast}{}{\N}\,\text{such that}\,\sigma\cdot\beta_{1}+\rho,\dots,\sigma\cdot\beta_{n}+\rho\in\prescript{\ast}{}{\N}.
\]

\end{document}